\documentclass[letterpaper,12pt]{article}
\usepackage[margin=1.2in]{geometry}
\usepackage{enumerate}
\usepackage{enumitem}

\usepackage{amssymb}
\usepackage{amsmath}
\usepackage{amsthm}
\usepackage{eufrak}
\usepackage{mathrsfs}

\usepackage{tikz}
\usetikzlibrary{calc}
\usepackage[labelformat=simple]{subcaption}
\renewcommand\thesubfigure{(\alph{subfigure})}
\definecolor{red1}{RGB}{230,25,75}
\definecolor{green1}{RGB}{60,180,75}

\usepackage{hyperref}
\usepackage{cleveref} 
\crefname{graph}{graph}{graphs}
\Crefname{graph}{Graph}{Graphs}

\newtheorem{theorem}{Theorem}[section]
\newtheorem{proposition}[theorem]{Proposition}

\newtheorem{lemma}[theorem]{Lemma}

\newtheorem{open}[theorem]{Open Problem}
\newtheorem{conjecture}[theorem]{Conjecture}

\usepackage[isbn=false,doi=false,style=numeric]{biblatex}

\begin{document}
\title{Strictly Metrizable Graphs are Minor-Closed}
\author{Maria Chudnovsky\thanks{Princeton University, Princeton, NJ 08544, USA. e-mail: mchudnov@math.princeton.edu. {Supported by  AFOSR grant FA9550-22-1-0083 and NSF Grant DMS-2348219. Part of this work was done  when the first author visited the Hebrew 
University of Jerusalem.} } \and{Daniel Cizma\thanks{Einstein Institute of Mathematics, Hebrew University, Jerusalem 91904, Israel. e-mail: daniel.cizma@mail.huji.ac.il.}} \and  {Nati Linial\thanks{School of Computer Science and Engineering, Hebrew University, Jerusalem 91904, Israel. e-mail: nati@cs.huji.ac.il.{~Supported in part by an ERC Grant 101141253, "Packing in Discrete Domains - Geometry and Analysis".}}}}

\maketitle

\begin{abstract}
    A {\em consistent path system} in a graph $G$ is an collection of paths, with exactly one path between any two vertices in $G$. A path system is said to be {\em consistent}
    if it is intersection-closed. We say that $G$ is {\em strictly metrizable} if every consistent path system in $G$ can be realized as the system of unique geodesics with respect to some assignment of positive edge weight. In this paper, we show that the family of strictly metrizable graphs is minor-closed.
\end{abstract}

\section{Introduction}
Let $G=(V,E)$ be a connected graph. Suppose you are given a sort of `road map' for 
it in the following sense: given any two vertices $u,v \in V$ this road map tells you the `chosen' path, $P_{u,v}$, between $u$ and $v$. This path need not be the shortest in any quantifiable
sense, e.g. w.r.t.\ edge count. Rather, this road map declares the path $P_{u,v}$ to be the best $(uv)$-path. A natural condition we would like this road map to have is consistency, namely, if $a\in P_{u,v}$ then $P_{u,v}$ should be the concatenation of $P_{u,a}$ and $P_{a,v}$. We call such `road maps' {\em consistent path systems}. 
Given such a path system it is suggestive to ask whether the paths in this
system can be made shortest in some measurable sense. More specifically, can the edges of $G$ be assigned positive weight so that the paths in our consistent path system are the {\em unique shortest paths} with respects to these edge weights.
We turn this to a question about graphs, and ask for which graphs $G$ it is the case that {\em every} consistent path system on $G$ is induced by some edge weights. 
Such graphs are said to be {\em strictly metrizable}. 

{\em Path complexes} and their corresponding homology theory were introduced
in \cite{GLMY}.
Building on this framework, Bodwin \cite{Bo}, has developed a structural theory of path systems which can be realized as the unique shortest paths in a weighted graph.
The study of strictly metrizable graphs was initiated in the paper \cite{CL}. 
It showed that strict metrizability is a rare property of graphs, but
nevertheless, there exist large non-trivial families of strictly metrizable graphs, e.g., outerplanar graphs. If we drop the requirement that the shortest paths be unique, we arrive
at the class of {\em metrizable} graphs. Further to \cite{CL}, such graphs were
investigated in \cite{CCL} where more detailed
structural information was derived. In particular, that paper shows that metrizable graphs have a very simple structure and are all nearly outerplanar.
In this paper we further analyze graph metrizability and show:
\begin{theorem}\label{thm:minor_closed}
The class of strictly metrizable graphs is minor-closed.
\end{theorem}

\subsection{Organization}
In \Cref{sec:prelim} we recall and develop some pertinent facts and tools.
\Cref{sec:results} is dedicated to our main results and proofs. Then,
we move to consider
edge weights that are not necessarily positive.
\Cref{sec:persistent_edges} explores the possibility of zero-weight edges.
Finally in \Cref{sec:open} we present some open problems, and ask in particular
(\Cref{open:neg}) how things change when negative edge weights are allowed.

\section{Preliminaries}\label{sec:prelim}
All graphs in this paper are finite and simple.  We recall some definitions and concepts from \cite{CL}. A {\em consistent path system} in a graph $G=(V,E)$ is a collection of paths $\mathcal{P}$ in $G$ satisfying two properties: 
 \begin{enumerate}[label={\arabic*)}]
     \item For every $u,v\in V$ there is exactly one $uv$-path in $\mathcal{P}$.
     The chosen $uv$-path and $vu$-path are identical, but in reverse. 
      \item The non-empty intersection of any two paths in $\mathcal{P}$ is a path in $\mathcal{P}$.
 \end{enumerate}
 A path system $\mathcal{P}$ of $G=(V,E)$ is {\em metric} if there is a positive weight function \mbox{$w:E\to \mathbb{R}_{>0}$} such that each path in $\mathcal{P}$ is a $w$-shortest path. Similarly, a path system $\mathcal{P}$ of $G=(V,E)$ is {\em strictly metric} if there is a positive weight function \mbox{$w:E\to \mathbb{R}_{>0}$} such that each path in $\mathcal{P}$ is the unique $w$-shortest path. We call $G$ {\em strictly metrizable} (s.m.), respectively metrizable, if every consistent path system in $G$ is strictly metric, respectively metric. It is known:
\begin{proposition}[\cite{CL}]
	\label{prop:top_minor}
Both the families of strictly metrizable graphs and metrizable graphs are closed under topological minors.   
\end{proposition}
\noindent
(Recall that $H$ is a topological minor of $G$ if $G$ contains a subdivision of $H$ as a subgraph.) Therefore, every graph that contains a subdivision of a non-s.m.\ graph is itself non-s.m. \\
A vertex is called a {\em branch vertex} if it has degree at least $3$. A {\em flat path}, aka a suspended path, in $G$ is a path 
whose internal vertices are of degree $2$ in $G$. We call an edge $xy$ in $G$ {\em compliant} if $x$ and $y$ are also connected by an additional flat path. As the following result shows, compliant edges have no effect on metrizability:
\begin{proposition}[\cite{CL}]\label{prop:compliant_can_be_deleted}
	If $e$ is a compliant edge in $G$, then $G$ is strictly metrizable if and only if 
    $G\setminus e$ is strictly metrizable.
\end{proposition}

{\em Halos} are useful in our analysis. Let $G=(V,E)$ be a graph and $x_1,x_2 \in V$. We call a subgraph $H$ of $G$ an {\em $x_1x_2$-halo} (\cref{fig:xy-halo}) if
\begin{enumerate}
	\item $H$ consists of a cycle $C$ along with two paths $P_1$ and $P_2$
	\item For $i=1,2$, the path $P_i$ is an $(x_i u_i)$-path with $u_i\in C$ and is otherwise disjoint from $C$. (We allow $P_i$ to consist of single vertex if $x_i \in C$).
	\item The paths $P_1$ and $P_2$ are vertex disjoint and their respective endpoints $u_1,u_2\in C$ are of distance at least $2$ in $C$.
\end{enumerate}
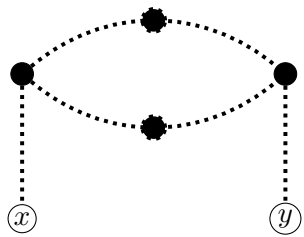
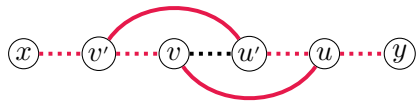
\begin{figure}[h]
	\centering
	\begin{subfigure}[t]{0.28\textwidth}
		\centering
		\begin{tikzpicture}
		\def\vtxSize{0.3cm}
		\def\r{1.75cm}
		\def\edgewidth{1.5pt}
		\node[draw,circle,minimum size=\vtxSize,inner sep=1pt,fill]  (1) at (-\r,0) {};
		\node[draw,circle,minimum size=\vtxSize,inner sep=1pt,fill]  (2) at (\r,0) {};
		\node[draw,circle,minimum size=\vtxSize,inner sep=1pt] (x) at (-\r,-1.1*\r) {\small$x$};
		\node[draw,circle,minimum size=\vtxSize,inner sep=1pt]  (y) at (\r,-1.1*\r) {\small$y$};

		\draw [line width=\edgewidth,dotted] (1) to[out=40, in=140]  node [pos=0.5,draw,circle,minimum size=\vtxSize,inner sep=1pt,fill]{} (2) ;
		\draw [line width=\edgewidth,dotted] (1) to[out=-40, in=-140] node [pos=0.5,draw,circle,minimum size=\vtxSize,inner sep=1pt,fill]{} (2);

		\draw [line width=\edgewidth,dotted] (1) -- (x) ;
		\draw [line width=\edgewidth,dotted] (2) -- (y);

		\end{tikzpicture}
		\caption{An $xy$-halo.}
		\label{fig:xy-halo}
		
	\end{subfigure}
		\hspace{30mm}
		\begin{subfigure}[t]{0.28\textwidth}
		\centering
		\begin{tikzpicture}
			\def\vtxSize{0.4cm}
		\def\r{2.5cm}
		\def\edgewidth{1.5pt}
		\node[draw,circle,minimum size=\vtxSize,inner sep=1pt]  (x) at (-\r,0) {\small$x$};
		\node[draw,circle,minimum size=\vtxSize,inner sep=1pt]  (y) at (\r,0) {\small$y$};
		\node[draw,circle,minimum size=\vtxSize,inner sep=1pt]  (v') at ($(x)!0.2!(y)$) {\small$v$\tiny$'$};
		\node[draw,circle,minimum size=\vtxSize,inner sep=1pt]  (v) at ($(x)!0.4!(y)$) {\small$v$};
		\node[draw,circle,minimum size=\vtxSize,inner sep=1pt]  (u') at ($(x)!0.6!(y)$) {\small$u$\tiny$'$};
		\node[draw,circle,minimum size=\vtxSize,inner sep=1pt]  (u) at ($(x)!0.8!(y)$) {\small$u$};

		\draw [line width=\edgewidth,dotted,red1] (x) -- (v') ;
		\draw [line width=\edgewidth,dotted,red1] (v') -- (v) ;
		\draw [line width=\edgewidth,dotted] (v) -- (u') ;
		\draw [line width=\edgewidth,dotted,red1] (u') -- (u) ;
		\draw [line width=\edgewidth,dotted,red1] (u) -- (y) ;
		\draw [line width=\edgewidth,red1] (v') to[out=55, in=125] (u') ;
		\draw [line width=\edgewidth,red1] (v) to[out=-55, in=-125] (u) ;
		
		\end{tikzpicture}
		\caption{The path $P$ along with the edges $uv$ and $u'v'$ form an $xy$-halo.}
		\label{fig:example_of_halo}
		
	\end{subfigure}
	\caption{Examples of halos in graphs.}
\end{figure}
We prove the following technical lemma concerning halos:
\begin{lemma}\label{lem:halo}
	Suppose that the graph $G=(V,E)$ is $2$-connected, and has no compliant edges. For $x,y\in V$, let $S$ be the vertex set of a connected component of $G-\{x,y\}$ and $H$ the subgraph induced by $S \cup \{x,y\}$. If there is a non-flat $(xy)$-path of length at least $2$ in $H$, then H has a subgraph that is an $xy$-Halo. 
\end{lemma}
\begin{proof}
	Let $P$ be a longest non-flat $(xy)$-path in $H$. First suppose that there exists a vertex $v\in H - P$. Since $G$ is $2$-connected, there exists two paths $Q_1$ and $Q_2$ from $v$ to $P$, which are disjoint except at $v$. Gluing these paths $Q_1$ and $Q_2$ at $v$, we obtain a path $P'$ with endpoints $z_1$ and $z_2$ in $P$. We observe $z_1$ and $z_2$ cannot be neighbors in $P$.  Otherwise, we can replace the edge $z_1z_2$ with $P'$ to obtain a strictly longer non-flat $xy$-path. Therefore, we may assume $z_1$ and $z_2$ are of distance at least $2$ in $P$. In this case, the union of $P$ and $P'$ form an $xy$-halo.\\
	So we may assume that $P$ contains all the vertices of $H$. By assumption, there is an internal vertex of $P$ which is a branch vertex. In particular, there exist vertices $u,v \in H$ such that $uv$ is an edge of $H$ but not of $P$.  
    Let us choose such $u$ and $v$ whose distance along $P$ is as small as possible, and let $Q_{u,v}$ be the $(uv)$-subpath of $P$. Since $uv$ is not a compliant edge, the path $Q_{u,v}$ cannot be flat. In particular, there exists an internal vertex $u'$
    of $Q_{u,v}$ with a neighbor $v'$ and $u'v'\notin Q_{u,v}$. Moreover, since we
    chose $u$ and $v$ as close as possible along $P$, it follows that $v'\notin Q_{u,v}$. But this implies that $P$ along with the edges $uv$ and $u'v'$ contain an $xy$-halo. 
    Indeed, if $v'$ is between $x$ and $v$ along $P$ then the $xv$ and $yu'$ subpaths of $P$ along with the edges $v'u'$ and $vu$ form an $xy$-halo, e.g. \cref{fig:example_of_halo}. Otherwise, $v'$ is between $y$ and $u$ along $P$ and the $yu$ and the $xu'$ subpaths of $P$ along with the edges $v'u'$ and $vu$ form an $xy$-halo.
\end{proof}
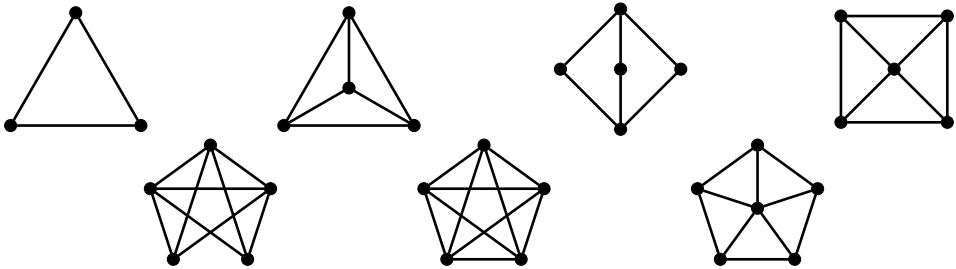
\begin{figure}[h]
    \centering
		\begin{minipage}[c]{0.12\textwidth}
			\centering
		\begin{tikzpicture}[scale=0.4, every node/.style={scale=0.4}]
			\def\vtxSize{.4cm}
			\def\edgeWidth{1pt}
			\def\radius{2.5cm}
			
			\node[draw,circle,minimum size=\vtxSize,inner sep=1pt,fill, color=black ](1) at (-0*360/3 +90: \radius) [scale=1] {};
			\node[draw,circle,minimum size=\vtxSize,inner sep=1pt,fill, color=black] (2) at (-1*360/3 +90: \radius) [scale=1]{};
			\node[draw,circle,minimum size=\vtxSize,inner sep=1pt,fill, color=black] (3) at (-2*360/3 +90: \radius)[scale=1] {};

			\draw [line width=\edgeWidth,-] (1) -- (2);
			\draw [line width=\edgeWidth,-] (2) -- (3);
			\draw [line width=\edgeWidth,-] (1) -- (3);

		\end{tikzpicture}
	\end{minipage}
        \hspace{15mm}
	\begin{minipage}[c]{0.12\textwidth}
		\centering
	\begin{tikzpicture}[scale=0.4, every node/.style={scale=0.4}]
		\def\vtxSize{.4cm}
		\def\edgeWidth{1pt}
		\def\radius{2.5cm}
		
		\node[draw,circle,minimum size=\vtxSize,inner sep=1pt,fill, color=black ](1) at (-0*360/3 +90: \radius) [scale=1] {};
		\node[draw,circle,minimum size=\vtxSize,inner sep=1pt,fill, color=black] (2) at (-1*360/3 +90: \radius) [scale=1]{};
		\node[draw,circle,minimum size=\vtxSize,inner sep=1pt,fill, color=black] (3) at (-2*360/3 +90: \radius)[scale=1] {};
		\node[draw,circle,minimum size=\vtxSize,inner sep=1pt,fill, color=black] (4) at (0,0)[scale=1] {};

		\draw [line width=\edgeWidth,-] (1) -- (2);
		\draw [line width=\edgeWidth,-] (1) -- (3);
		\draw [line width=\edgeWidth,-] (1) -- (4);
		\draw [line width=\edgeWidth,-] (2) -- (3);
		\draw [line width=\edgeWidth,-] (2) -- (4);
		\draw [line width=\edgeWidth,-] (3) -- (4);
		
	\end{tikzpicture}
\end{minipage}
\hspace{15mm}
\begin{minipage}[c]{0.12\textwidth}
	\centering
\begin{tikzpicture}[scale=0.4, every node/.style={scale=0.4}]
	\def\vtxSize{.4cm}
	\def\edgeWidth{1pt}
	\def\radius{2.5cm}
	
	\node[draw,circle,minimum size=\vtxSize,inner sep=1pt,fill, color=black] (1) at (-2,0) [scale=1] {};
	\node[draw,circle,minimum size=\vtxSize,inner sep=1pt,fill, color=black] (2) at (0,0) [scale=1]{};
	\node[draw,circle,minimum size=\vtxSize,inner sep=1pt,fill, color=black] (3) at (2,0)[scale=1] {};
	\node[draw,circle,minimum size=\vtxSize,inner sep=1pt,fill, color=black] (4) at (0,-2)[scale=1] {};
	\node[draw,circle,minimum size=\vtxSize,inner sep=1pt,fill, color=black] (5) at (0,2)[scale=1] {};

	\draw [line width=\edgeWidth,-] (1) -- (4);
	\draw [line width=\edgeWidth,-] (2) -- (4);
	\draw [line width=\edgeWidth,-] (3) -- (4);
	\draw [line width=\edgeWidth,-] (1) -- (5);
	\draw [line width=\edgeWidth,-] (2) -- (5);
	\draw [line width=\edgeWidth,-] (3) -- (5);

\end{tikzpicture}
\end{minipage}
\hspace{15mm}
\begin{minipage}[c]{0.12\textwidth}
	\centering
\begin{tikzpicture}[scale=0.4, every node/.style={scale=0.4}]
	\def\vtxSize{.4cm}
	\def\edgeWidth{1pt}
	\def\radius{2.5cm}
	
	\node[draw,circle,minimum size=\vtxSize,inner sep=1pt,fill, color=black ](1) at (-0*360/4 +45 : \radius) [scale=1] {};
	\node[draw,circle,minimum size=\vtxSize,inner sep=1pt,fill, color=black] (2) at (-1*360/4 +45: \radius) [scale=1]{};
	\node[draw,circle,minimum size=\vtxSize,inner sep=1pt,fill, color=black] (3) at (-2*360/4 +45: \radius)[scale=1] {};
	\node[draw,circle,minimum size=\vtxSize,inner sep=1pt,fill, color=black] (4) at (-3*360/4 +45: \radius)[scale=1] {};
	\node[draw,circle,minimum size=\vtxSize,inner sep=1pt,fill, color=black] (5) at (0,0)[scale=1] {};

	\draw [line width=\edgeWidth,-] (1) -- (2);
	\draw [line width=\edgeWidth,-] (1) -- (4);
	\draw [line width=\edgeWidth,-] (1) -- (5);
	\draw [line width=\edgeWidth,-] (2) -- (3);
	\draw [line width=\edgeWidth,-] (2) -- (5);
	\draw [line width=\edgeWidth,-] (3) -- (4);
	\draw [line width=\edgeWidth,-] (3) -- (5);
	\draw [line width=\edgeWidth,-] (4) -- (5);
	
\end{tikzpicture}
\end{minipage}
\\
\begin{minipage}[c]{0.12\textwidth}
	\centering
\begin{tikzpicture}[scale=0.4, every node/.style={scale=0.4}]
	\def\vtxSize{.4cm}
	\def\edgeWidth{1pt}
	\def\radius{2.1cm}
	
	\node[draw,circle,minimum size=\vtxSize,inner sep=1pt,fill, color=black](1) at (0*360/5 +162 : \radius) [scale=1] {};
	\node[draw,circle,minimum size=\vtxSize,inner sep=1pt,fill, color=black] (2) at (1*360/5 +162: \radius) [scale=1]{};
	\node[draw,circle,minimum size=\vtxSize,inner sep=1pt,fill, color=black] (3) at (2*360/5 +162: \radius)[scale=1] {};
	\node[draw,circle,minimum size=\vtxSize,inner sep=1pt,fill, color=black] (4) at (3*360/5 +162: \radius)[scale=1] {};
	\node[draw,circle,minimum size=\vtxSize,inner sep=1pt,fill, color=black] (5) at (4*360/5 +162: \radius)[scale=1] {};

	\draw [line width=\edgeWidth,-] (1) -- (2);
	\draw [line width=\edgeWidth,-] (1) -- (3);
	\draw [line width=\edgeWidth,-] (1) -- (4);
	\draw [line width=\edgeWidth,-] (1) -- (5);
	\draw [line width=\edgeWidth,-] (2) -- (4);
	\draw [line width=\edgeWidth,-] (2) -- (5);
	\draw [line width=\edgeWidth,-] (3) -- (4);
	\draw [line width=\edgeWidth,-] (3) -- (5);
	\draw [line width=\edgeWidth,-] (4) -- (5);
	
	\end{tikzpicture}
\end{minipage}
\hspace{15mm}
\begin{minipage}[c]{0.12\textwidth}
	\centering
\begin{tikzpicture}[scale=0.4, every node/.style={scale=0.4}]
	\def\vtxSize{.4cm}
	\def\edgeWidth{1pt}
	\def\radius{2.1cm}
	
	\node[draw,circle,minimum size=\vtxSize,inner sep=1pt,fill, color=black](1) at (0*360/5 +162 : \radius) [scale=1] {};
	\node[draw,circle,minimum size=\vtxSize,inner sep=1pt,fill, color=black] (2) at (1*360/5 +162: \radius) [scale=1]{};
	\node[draw,circle,minimum size=\vtxSize,inner sep=1pt,fill, color=black] (3) at (2*360/5 +162: \radius)[scale=1] {};
	\node[draw,circle,minimum size=\vtxSize,inner sep=1pt,fill, color=black] (4) at (3*360/5 +162: \radius)[scale=1] {};
	\node[draw,circle,minimum size=\vtxSize,inner sep=1pt,fill, color=black] (5) at (4*360/5 +162: \radius)[scale=1] {};

	\draw [line width=\edgeWidth,-] (1) -- (2);
	\draw [line width=\edgeWidth,-] (1) -- (3);
	\draw [line width=\edgeWidth,-] (1) -- (4);
	\draw [line width=\edgeWidth,-] (1) -- (5);
	\draw [line width=\edgeWidth,-] (2) -- (3);
	\draw [line width=\edgeWidth,-] (2) -- (4);
	\draw [line width=\edgeWidth,-] (2) -- (5);
	\draw [line width=\edgeWidth,-] (3) -- (4);
	\draw [line width=\edgeWidth,-] (3) -- (5);
	\draw [line width=\edgeWidth,-] (4) -- (5);
	
	\end{tikzpicture}
\end{minipage}
\hspace{15mm}
\begin{minipage}[c]{0.12\textwidth}
	\centering
\begin{tikzpicture}[scale=0.4, every node/.style={scale=0.4}]
	\def\vtxSize{.4cm}
	\def\edgeWidth{1pt}
	\def\radius{2.1cm}
	
	\node[draw,circle,minimum size=\vtxSize,inner sep=1pt,fill, color=black ](1) at (-0*360/5 +90 : \radius) [scale=1] {};
	\node[draw,circle,minimum size=\vtxSize,inner sep=1pt,fill, color=black] (2) at (-1*360/5 +90: \radius) [scale=1]{};
	\node[draw,circle,minimum size=\vtxSize,inner sep=1pt,fill, color=black] (3) at (-2*360/5 +90: \radius)[scale=1] {};
	\node[draw,circle,minimum size=\vtxSize,inner sep=1pt,fill, color=black] (4) at (-3*360/5 +90: \radius)[scale=1] {};
	\node[draw,circle,minimum size=\vtxSize,inner sep=1pt,fill, color=black] (5) at (-4*360/5 +90: \radius)[scale=1] {};
	\node[draw,circle,minimum size=\vtxSize,inner sep=1pt,fill, color=black] (6) at (0,0)[scale=1] {};

	\draw [line width=\edgeWidth,-] (1) -- (2);
	\draw [line width=\edgeWidth,-] (1) -- (5);
	\draw [line width=\edgeWidth,-] (1) -- (6);
	\draw [line width=\edgeWidth,-] (2) -- (3);
	\draw [line width=\edgeWidth,-] (2) -- (6);
	\draw [line width=\edgeWidth,-] (3) -- (4);
	\draw [line width=\edgeWidth,-] (3) -- (6);
	\draw [line width=\edgeWidth,-] (4) -- (5);
	\draw [line width=\edgeWidth,-] (4) -- (6);
	\draw [line width=\edgeWidth,-] (5) -- (6);

\end{tikzpicture}
\end{minipage}
    \caption{Up to adding compliant edges, all $2$-connected strictly metrizable graphs are subdivisions of one of the above graphs.}
    \label{fig:mainTheorem}
\end{figure}
\section{Main Results}\label{sec:results}
In order to prove that strictly metrizable graphs are minor-closed, we need the following theorem which describes the form of such graphs. (recall: $W_4'$ is the $5$-vertex graph with $9$ edges). 
\begin{theorem} \label{thm:structure}
    Every $2$-connected s.m.\ graph with
	no compliant edges is either $K_5$, $W_5$ or a subdivision of one of the following: $K_{2,3}$, $K_4$, $W_4$ or $W_4'$.
\end{theorem}
We note that the statement and proof of this theorem are nearly identical to those of
a theorem from \cite{CCL} on the structure of (not necessarily strictly) metrizable graphs of order at least $11$. In contrast, \cref{thm:structure} holds for {\em strictly metrizable graphs of any order}.
The proof of \cref{thm:structure} relies on the following proposition, which substantially restricts the structure of strictly metrizable graphs. Again, a nearly identical proposition in \cite{CCL} speaks about (not necessarily strictly) metrizable graphs.
\begin{proposition} \label{prop:disjoint_cycles}
Suppose that the graph $G$ is $2$-connected,  has no compliant edges and contains two disjoint cycles $C_1$ and $C_2$. Then $G$ is not strictly metrizable.
\end{proposition}
\begin{proof}
        We will show that $G$ contains a subgraph which is a subdivision of one of graphs in \cref{fig:zoo}. From \cref{prop:top_minor} in then follows that $G$ is not strictly metrizable.
	Since $G$ is $2$-connected, there exists two disjoint paths, $R_1$ and 
	$R_2$, (possibly edges) between $C_1$ and $C_2$. Say $R_1$ connects between $x_1\in C_1$ and $x_2\in C_2$, and $R_2$ connects $y_1\in C_1$ and $y_2\in C_2$.
	Set $H = C_1\cup C_2 \cup R_1\cup R_2$, \cref{fig:disCycCase1_1}.
	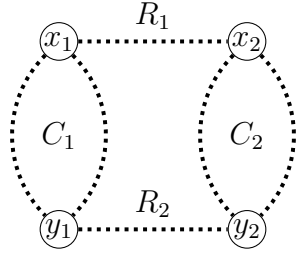
\begin{figure}[h]
		\centering
			\begin{tikzpicture}
			\def\vtxSize{0.5cm}
                \def\width{1.25cm}
			\def\height{1.25cm}
                \def\edgewidth{1.5pt}
			\node[draw,circle,minimum size=\vtxSize,inner sep=0pt] (x1) at (-\width,\height) {$x_1$};
                \node[draw,circle,minimum size=\vtxSize,inner sep=0pt] (y1) at (-\width,-\height) {$y_1$};
                \node[draw,circle,minimum size=\vtxSize,inner sep=0pt] (x2) at  (\width, \height){$x_2$};
                \node[draw,circle,minimum size=\vtxSize,inner sep=0pt] (y2) at  (\width, -\height){$y_2$};
                \node at (-\width, 0) {$C_1$};
                \node at (\width, 0) {$C_2$};

			\draw [line width=\edgewidth,dotted] (x1) -- (x2) node[midway, above]{$R_1$};
                \draw [line width=\edgewidth,dotted] (y1) -- (y2) node[midway, above]{$R_2$};;
                \draw [line width=\edgewidth,dotted] (x1) to[out=-135, in=135] (y1);
                \draw [line width=\edgewidth,dotted] (x1) to[out=-45, in=45] (y1);
                \draw [line width=\edgewidth,dotted] (x2) to[out=-135, in=135] (y2);
                \draw [line width=\edgewidth,dotted] (x2) to[out=-45, in=45] (y2);

			\end{tikzpicture}
			\caption{The subgraph $H$ of disjoint cycles.}
			\label{fig:disCycCase1_1}
		
	\end{figure}
	There are three cases to consider: 
	\begin{itemize}
		\item[]  Case 1. Neither $x_1y_1$ nor $x_2y_2$ are edges of $C_1$ resp.\ $C_2$.
		\item[] Case 2. $x_1,y_1$ are not adjacent in $C_1$ but $x_2y_2$ is an edge of $C_2$
		\item[]  Case 3. Both $x_1y_1$ and $x_2y_2$ are adjacent in $C_1$ and $C_2$, respectively
	\end{itemize}

	\begin{figure}[h]
		\centering
		\begin{subfigure}[t]{0.28\textwidth}
			\centering
			\begin{tikzpicture}
			\def\vtxSize{0.5cm}
				\def\width{1.25cm}
			\def\height{1.25cm}
				\def\edgewidth{1.5pt}
				\def\r{\width/2}
			\node[draw,circle,minimum size=\vtxSize,inner sep=0pt] (x1) at (-\width,\height) {$x_1$};
				\node[draw,circle,minimum size=\vtxSize,inner sep=0pt] (y1) at (-\width,-\height) {$y_1$};
				\node[draw,circle,minimum size=\vtxSize,inner sep=0pt] (x2) at  (\width, \height){$x_2$};
				\node[draw,circle,minimum size=\vtxSize,inner sep=0pt] (y2) at  (\width, -\height){$y_2$};
			\node[draw,circle,minimum size=.5*\vtxSize,inner sep=1pt,fill, color=black] (z) at  ($(x1)!.5!(x2)$){};
				\node[draw,circle,minimum size=\vtxSize,inner sep=0pt] (z') at  ($(x2)!.5!(y2) - (\r,0)$){$u$};

			\draw [line width=\edgewidth,dotted,red1] (x1) -- (z) ;
				\draw [line width=\edgewidth,dotted,red1] (x2) -- (z) ;
				\draw [line width=\edgewidth,dotted,red1] (y1) -- (y2);
				\draw [line width=\edgewidth,dotted,red1] (x1) to[out=-135, in=135] (y1);
				\draw [line width=\edgewidth,dotted,red1] (x1) to[out=-45, in=45] (y1);
				\draw [line width=\edgewidth,dotted] (x2) to[out=-135, in=90] (z');
				\draw [line width=\edgewidth,dotted,red1] (z') to[out=-90, in=135] (y2);
				\draw [line width=\edgewidth,-,red1] (x2) to[out=-45, in=45]   (y2);
				\draw [line width=\edgewidth,-,red1] (z) to (z');
	  

			\end{tikzpicture}
			\caption{A path from $Q$ to $R_1$ recovers the previous case.}
			\label{fig:disCycCase2_1}
			
		\end{subfigure}
			\hfill 
			\begin{subfigure}[t]{0.28\textwidth}
			\centering
			\begin{tikzpicture}
			\def\vtxSize{0.5cm}
				\def\width{1.25cm}
			\def\height{1.25cm}
				\def\edgewidth{1.5pt}
				\def\r{\width/2}
			\node[draw,circle,minimum size=\vtxSize,inner sep=0pt] (x1) at (-\width,\height) {$x_1$};
				\node[draw,circle,minimum size=\vtxSize,inner sep=0pt] (y1) at (-\width,-\height) {$y_1$};
				\node[draw,circle,minimum size=\vtxSize,inner sep=0pt] (x2) at  (\width, \height){$x_2$};
				\node[draw,circle,minimum size=\vtxSize,inner sep=0pt] (y2) at  (\width, -\height){$y_2$};

				\node[draw,circle,minimum size=\vtxSize,inner sep=0pt] (z') at  ($(x2)!.5!(y2) - (\r,0)$){$u$};

			\draw [line width=\edgewidth,dotted,red1] (x1) -- (x2) ;
				\draw [line width=\edgewidth,dotted,red1] (y1) -- (y2);
				\draw [line width=\edgewidth,dotted,red1] (x1) to[out=-135, in=135] (y1);
				\draw [line width=\edgewidth,dotted,red1] (x1) to[out=-45, in=45] (y1);
				\draw [line width=\edgewidth,dotted] (x2) to[out=-135, in=90] (z');
				\draw [line width=\edgewidth,dotted,red1] (z') to[out=-90, in=135] (y2);
				\draw [line width=\edgewidth,-,red1] (x2) to[out=-45, in=45]   (y2);
				\draw [line width=\edgewidth,-,red1] (x1) to (z');
				\node[draw,circle,minimum size=.5*\vtxSize,inner sep=1pt,fill, color=black] (v) at  ($(x1)!.5!(y1) - (\r,0)$){};
				\node[draw,circle,minimum size=.5*\vtxSize,inner sep=1pt,fill, color=black] (v') at  ($(x1)!.5!(y1) + (\r,0)$){};

			\end{tikzpicture}
			\caption{A path from $P$ to $x_1$ gives a subdivision of \cref{fig:graph6}.}
			\label{fig:disCycCase2_2}
			
		\end{subfigure}
		\hfill
		\begin{subfigure}[t]{0.28\textwidth}
			\centering
			\begin{tikzpicture}
			\def\vtxSize{0.5cm}
				\def\width{1.25cm}
			\def\height{1.25cm}
				\def\edgewidth{1.5pt}
				\def\r{\width/2}
			\node[draw,circle,minimum size=\vtxSize,inner sep=0pt] (x1) at (-\width,\height) {$x_1$};
				\node[draw,circle,minimum size=\vtxSize,inner sep=0pt] (y1) at (-\width,-\height) {$y_1$};
				\node[draw,circle,minimum size=\vtxSize,inner sep=0pt] (x2) at  (\width, \height){$x_2$};
				\node[draw,circle,minimum size=\vtxSize,inner sep=0pt] (y2) at  (\width, -\height){$y_2$};

				\node[draw,circle,minimum size=\vtxSize,inner sep=0pt] (z') at  ($(x2)!.5!(y2) - (\r,0)$){$u$};

			\draw [line width=\edgewidth,dotted,red1] (x1) -- (x2) ;
				\draw [line width=\edgewidth,dotted,red1] (y1) -- (y2);
				\draw [line width=\edgewidth,dotted,red1] (x1) to[out=-135, in=135] (y1);
				\draw [line width=\edgewidth,dotted,red1] (x1) to[out=-45, in=45] (y1);
				\draw [line width=\edgewidth,dotted] (x2) to[out=-135, in=90] (z');
				\draw [line width=\edgewidth,dotted,red1] (z') to[out=-90, in=135] (y2);
				\draw [line width=\edgewidth,-,red1] (x2) to[out=-45, in=45]   (y2);
				
				\node[draw,circle,minimum size=.5*\vtxSize,inner sep=1pt,fill, color=black] (v) at  ($(x1)!.5!(y1) - (\r,0)$){};
				\node[draw,circle,minimum size=.5*\vtxSize,inner sep=1pt,fill, color=black] (v') at  ($(x1)!.5!(y1) + (\r,0)$){};

				\draw [line width=\edgewidth,-,red1] (v') to (z');

			\end{tikzpicture}
			\caption{A path from $P$ to $C_1-\{x_1,y_2\}$ gives a subdivision of \cref{fig:graph2}.}
			\label{fig:disCycCase2_3}
			
		\end{subfigure}
		\caption{The various scenarios of Case 2, in the proof of \cref{prop:disjoint_cycles}. }
	\end{figure}
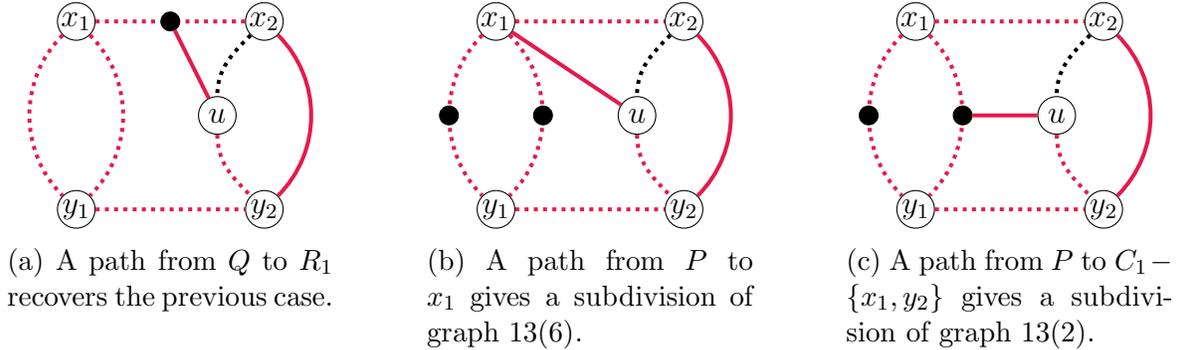
 \noindent 
	{\it Case 1:}
	In this case, $H$ is simply a subdvision of \cref{fig:graph8}.\\
	\noindent
	{\it Case 2:} Here we view $C_2$ as the union of an $(x_2y_2)$-path $P$, $|P|\geq 2$, 
and the edge $x_2y_2$, which, by assumption is non-compliant. The discussion proceeds now
according to whether $x_2,y_2$ separate $C_2$ from $H - C_2$ or not. If they do, then
by \cref{lem:halo} the component of $H-\{x_2,y_2\}$ containing $C_2$ must contain a $x_2y_2$-halo. Considering the union of this halo and $H-C_2$ we are back to Case 1 which is already settled. \\
Now we may assume that $x_2,y_2$ do not disconnect $C_2$ from the rest of $H$. In particular, there exists a path $Q$ (possibly an edge) from an internal vertex, $u$, of $P$ to a vertex $v$ in $H-P$ which is otherwise disjoint from $H$. Up to symmetries there are three possibilities to consider: $Q$ connects between $u$ and

\begin{enumerate}[label=\alph*]
    \item 
\hspace{-0.08in}) an internal vertex of $R_1$
    \item 
\hspace{-0.08in}) $x_1$
    \item
\hspace{-0.08in}) $C_1-\{x_1,y_1\}$.
\end{enumerate}

In a), the situation reduces to Case 1, \cref{fig:disCycCase2_1}, which is already settled. In b), G contains a subdivision of \cref{fig:graph6}, see \cref{fig:disCycCase2_2}. In c), $G$ contains a subdivision of \cref{fig:graph2}, \cref{fig:disCycCase2_3}.\\
	\noindent
	{\it Case 3:} Now each $C_i$ is the union of the edge $x_iy_i$ and an $(x_iy_i)$-path 
	$P_i$ of length at least $2$. First assume that $x_2,y_2$ separate $C_2$ from $C_1$. As argued above, by \cref{lem:halo} the component of $H-\{x_2,y_2\}$ containing $C_2$ must contain a $x_2y_2$-halo. As before taking the union of this halo and $H-C_2$ we are back to Case 2 which is settled. \\
	So we may assume that $x_2,y_2$  do not separate $C_2$ from $C_1$. This means that there is a path $Q$ (possibly an edge) connecting a vertex $u$ in $C_2-\{x_2,y_2\}$, i.e. an internal vertex of $P_2$, to a vertex in $H-C_2$. First assume that $Q$ is a path from $u$ to an internal vertex of $R_1$ (or $R_2$). 
	(Recall that $R_1$ and $R_2$ are disjoint paths between $C_1$ and $C_2$, 
	\cref{fig:disCycCase1_1}.)  
	This creates the same situation as in Case 2, see e.g. \cref{fig:disCycCase2_1}.
	
	So we may assume that $Q$ is a path from $u$ to $C_1$. Next suppose that $Q$ is a path connecting $u$ to an internal vertex of $P_1$. This gives us a
	subdivision of \cref{fig:graph2}, \cref{fig:disCycCase3_1}.
	We can therefore assume that the path $Q$ connects $u$ to the set $\{x_1,y_1\}$. The above arguments
	allow us to assume as well that $x_1,y_1$ do not separate $C_1$ from $C_2$, and that there is a path $Q'$
	from a vertex $u'\in C_1-\{x_1,y_1\}$ to the set $\{x_2,y_2\}$, which is internally disjoint from $H$. 
	There are now, up to symmetries, 2 subcases to consider: a) $Q$ is a $(ux_1)$-path and $Q'$ is a $(u'x_2)$-path, b) $Q$ is a $(ux_1)$-path and $Q'$ is $(u'y_2)$ path. In subcase a), $G$ contains a 
	subdivision of the prism, \cref{fig:disCycCase3_2}. In subcase b), notice that this graph contains a $K_{2,4}$, \cref{fig:graph1}, and is therefore not s.m., \cref{fig:disCycCase3_3}.
	
	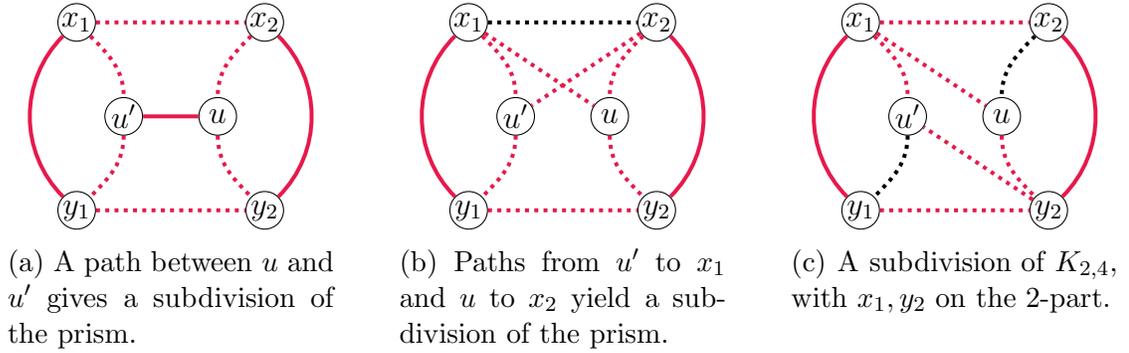
\begin{figure}
		\centering
		
		\hfill
            \begin{subfigure}[t]{0.28\textwidth}
			\centering
			\begin{tikzpicture}
			\def\vtxSize{0.5cm}
                \def\width{1.25cm}
			\def\height{1.25cm}
                \def\edgewidth{1.5pt}
                \def\r{\width/2}
			\node[draw,circle,minimum size=\vtxSize,inner sep=0pt] (x1) at (-\width,\height) {$x_1$};
                \node[draw,circle,minimum size=\vtxSize,inner sep=0pt] (y1) at (-\width,-\height) {$y_1$};
                \node[draw,circle,minimum size=\vtxSize,inner sep=0pt] (x2) at  (\width, \height){$x_2$};
                \node[draw,circle,minimum size=\vtxSize,inner sep=0pt] (y2) at  (\width, -\height){$y_2$};
		  \node[draw,circle,minimum size=\vtxSize,inner sep=0pt] (u1) at  ($(x1)!.5!(y1) + (\r,0)$){$u'$};
                \node[draw,circle,minimum size=\vtxSize,inner sep=0pt] (u2) at  ($(x2)!.5!(y2) - (\r,0)$){$u$};

			\draw [line width=\edgewidth,dotted,red1] (x1) -- (x2) ;
                \draw [line width=\edgewidth,dotted,red1] (y1) -- (y2);
                \draw [line width=\edgewidth,dotted,red1] (x1) to[out=-45, in=90] (u1);
                \draw [line width=\edgewidth,dotted,red1] (u1) to[out=-90, in=45] (y1);
                \draw [line width=\edgewidth,red1] (x1) to[out=-135, in=135] (y1);
                \draw [line width=\edgewidth,dotted,red1] (x2) to[out=-135, in=90] (u2);
                \draw [line width=\edgewidth,dotted,red1] (u2) to[out=-90, in=135] (y2);
                \draw [line width=\edgewidth,-,red1] (x2) to[out=-45, in=45]   (y2);
                \draw [line width=\edgewidth,-,red1] (u1) -- (u2);

			\end{tikzpicture}
			\caption{A path between $u$ and $u'$ gives a subdivision of the prism.}
			\label{fig:disCycCase3_1}
			
		\end{subfigure}
            \hfill
		\begin{subfigure}[t]{0.28\textwidth}
			\centering
			\begin{tikzpicture}
			\def\vtxSize{0.5cm}
                \def\width{1.25cm}
			\def\height{1.25cm}
                \def\edgewidth{1.5pt}
                \def\r{\width/2}
			\node[draw,circle,minimum size=\vtxSize,inner sep=0pt] (x1) at (-\width,\height) {$x_1$};
                \node[draw,circle,minimum size=\vtxSize,inner sep=0pt] (y1) at (-\width,-\height) {$y_1$};
                \node[draw,circle,minimum size=\vtxSize,inner sep=0pt] (x2) at  (\width, \height){$x_2$};
                \node[draw,circle,minimum size=\vtxSize,inner sep=0pt] (y2) at  (\width, -\height){$y_2$};
		  \node[draw,circle,minimum size=\vtxSize,inner sep=0pt] (u1) at  ($(x1)!.5!(y1) + (\r,0)$){$u'$};
                \node[draw,circle,minimum size=\vtxSize,inner sep=0pt] (u2) at  ($(x2)!.5!(y2) - (\r,0)$){$u$};

			\draw [line width=\edgewidth,dotted] (x1) -- (x2) ;
                \draw [line width=\edgewidth,dotted,red1] (y1) -- (y2);
                \draw [line width=\edgewidth,dotted,red1] (x1) to[out=-45, in=90] (u1);
                \draw [line width=\edgewidth,dotted,red1] (u1) to[out=-90, in=45] (y1);
                \draw [line width=\edgewidth, red1] (x1) to[out=-135, in=135] (y1);
                \draw [line width=\edgewidth,dotted,red1] (x2) to[out=-135, in=90] (u2);
                \draw [line width=\edgewidth,dotted,red1] (u2) to[out=-90, in=135] (y2);
                \draw [line width=\edgewidth,-,red1] (x2) to[out=-45, in=45]   (y2);
                \draw [line width=\edgewidth,-,red1, dotted] (x1) to (u2);
                \draw [line width=\edgewidth,-,red1, dotted] (x2) to (u1);

			\end{tikzpicture}
			\caption{Paths from $u'$ to $x_1$ and $u$ to $x_2$ yield a subdivision of the prism.}
			\label{fig:disCycCase3_2}
			
		\end{subfigure}
            \hfill 
            \begin{subfigure}[t]{0.28\textwidth}
			\centering
			\begin{tikzpicture}
			\def\vtxSize{0.5cm}
                \def\width{1.25cm}
			\def\height{1.25cm}
                \def\edgewidth{1.5pt}
                \def\r{\width/2}
			\node[draw,circle,minimum size=\vtxSize,inner sep=0pt] (x1) at (-\width,\height) {$x_1$};
                \node[draw,circle,minimum size=\vtxSize,inner sep=0pt] (y1) at (-\width,-\height) {$y_1$};
                \node[draw,circle,minimum size=\vtxSize,inner sep=0pt] (x2) at  (\width, \height){$x_2$};
                \node[draw,circle,minimum size=\vtxSize,inner sep=0pt] (y2) at  (\width, -\height){$y_2$};
		  \node[draw,circle,minimum size=\vtxSize,inner sep=0pt] (u1) at  ($(x1)!.5!(y1) + (\r,0)$){$u'$};
                \node[draw,circle,minimum size=\vtxSize,inner sep=0pt] (u2) at  ($(x2)!.5!(y2) - (\r,0)$){$u$};

			\draw [line width=\edgewidth,dotted,red1] (x1) -- (x2) ;
                \draw [line width=\edgewidth,dotted,red1] (y1) -- (y2);
                \draw [line width=\edgewidth,dotted,red1] (x1) to[out=-45, in=90] (u1);
                \draw [line width=\edgewidth,dotted] (u1) to[out=-90, in=45] (y1);
                \draw [line width=\edgewidth, red1] (x1) to[out=-135, in=135] (y1);
                \draw [line width=\edgewidth,dotted] (x2) to[out=-135, in=90] (u2);
                \draw [line width=\edgewidth,dotted,red1] (u2) to[out=-90, in=135] (y2);
                \draw [line width=\edgewidth,-,red1] (x2) to[out=-45, in=45]   (y2);
                \draw [line width=\edgewidth,-,red1, dotted] (x1) to (u2);
                \draw [line width=\edgewidth,-,red1, dotted] (y2) to (u1);

			\end{tikzpicture}
			\caption{A subdivision of $K_{2,4}$, with $x_1, y_2$ on the $2$-part.}
			\label{fig:disCycCase3_3}
			
		\end{subfigure}
		\caption{The subcases of Case 3, in the proof of \cref{prop:disjoint_cycles}.}
	\end{figure}

	\end{proof}

In order to prove \cref{thm:structure} we need the following lemma.
Consider the graphs obtained by adding $1$, $2$ and $3$ edges to the side of $3$ vertices in $K_{3,n}$. Call them
$K_{3,n}'$, $K_{3,n}''$ and $K_{3,n}'''$, respectively.
\begin{lemma}[Lov\'asz, \cite{L}]\label{lem:Lovasz}
If a graph $G$ with $\delta(G)\geq 3$ contains no disjoint cycles, 
then it is either $K_5$, $W_n$, $K_{3,n}$, $K_{3,n}'$, $K_{3,n}''$ or $K_{3,n}'''$.
\end{lemma}
\begin{proof}{[\Cref{thm:structure}]} 
Let $G$ be an  $2$-connected s.m.\ graph with no compliant edges.
We may assume that $G$ does not contain disjoint cycles by \cref{prop:disjoint_cycles}.\\ 
In order to use \cref{lem:Lovasz}, we need to eliminate all vertices of degree $2$ in $G$, 
since the lemma assumes $\delta(G)\geq 3$. If we suppress all vertices of degree $2$ in $G$,
the resulting graph $\tilde{G}$ indeed satisfies $\delta(\tilde{G})\geq 3$, though
it need not be a {\em simple} graph. Clearly, $\tilde{G}$ is
simple if and only if $G$ contains no parallel flat paths, so let us consider
what happens if $G$ contains parallel flat paths. Namely, there are at least two flat $(uv)$-paths in $G$ between a pair of
vertices $u,v\in G$. We show that there cannot be an additional vertex pair $\{x,y\} \neq \{u,v\}$ that is connected by parallel flat paths. Indeed,
$\{x,y\}$, $\{u,v\}$ cannot be disjoint, or else
these four flat paths form two disjoint cycles. So suppose that there are parallel flat paths between $x,y$ as well as between $x,v$. 
Since $G$ is $2$-connected, there must be another path from $v$ to $y$ disjoint from $x$. This $(vy)$-path must also be internally disjoint from the parallel flat paths between  $x,y$ and $x,v$ (because they are flat).  But this forms a subdivision of \cref{fig:graph6}.\\
So let $x,y$ be the unique pair of vertices between which there are parallel flat paths. 
If $x$ and $y$ are the only branch vertices of $G$
then $G$ is a subdivision of $K_{2,n}$, $n\geq 3$. Since $K_{2,4}$, \cref{fig:graph1}, is not s.m., we must have $n=3$ and $G$ is a subdivided $K_{2,3}$.\\ 
Now we may assume that $G$ has another branch vertex $z\neq x, y$. Since $G$
is $2$-connected, $z$ must reside on some $(xy)$-path. By \cref{lem:halo} the connected component of $G-\{x,y\}$ containing $z$ admits an $xy$-halo. Let us consider
this $xy$-halo, along with the other parallel flat $(xy)$-paths. There are
3 possibilities to consider.
The graph $G$ contains a subdivision of either \cref{fig:graph1}, \cref{fig:graph6} or 
\cref{fig:graph8}, see \cref{fig:parrallel_plus_halo}. In each case $G$ is not metrizable.
In the eventual case $G$ has no parallel flat paths. By
suppressing all its degree-$2$ vertices, we obtain a simple graph $\tilde{G}$ with $\delta(\tilde{G})\geq 3$, and 
by \Cref{lem:Lovasz}, $\tilde{G}$ is either $K_5$, $W_n$, $K_{3,n}$, $K_{3,n}'$, $K_{3,n}''$ or $K_{3,n}'''$. But  $K_{3,n}$, $K_{3,n}'$, $K_{3,n}''$ and $K_{3,n}'''$ are not s.m.\ for $n\geq 3$, since they contain $K_{3,3}$, \cref{fig:graph3}. The wheels
$W_n$ for $n\geq 6$ are excluded, since they contain a subdivision of \cref{fig:graph4}. The possibilities for $\tilde{G}$ that remain are precisely  $K_4$, $K_5$, $W_4$, $W_5$ and $W_4'$. Finally, we note that every strict subdivision of $K_5$ contains a subdivision of $K_{2,4}$, and any strict subdivision of $W_5$ contains a subdivision of either \cref{fig:graph4} or \cref{fig:graph5}.

\end{proof}

	\begin{figure}[h]
		\centering
		\begin{subfigure}[t]{0.28\textwidth}
			\centering
			\begin{tikzpicture}
			\def\vtxSize{0.5cm}
				\def\width{1.25cm}
			\def\height{1.25cm}
				\def\edgewidth{1.5pt}
				\def\s{2.5pt}
			\node[draw,circle,minimum size=\vtxSize,inner sep=0pt] (x) at (-\width,\height) {$x$};
				\node[draw,circle,minimum size=\vtxSize,inner sep=0pt] (y) at (-\width,-\height) {$y$};
				\node[draw,circle,fill,inner sep=\s] (x') at  (\width, \height){};
				\node[draw,circle,fill,inner sep=\s] (y') at  (\width, -\height){};

			\draw [line width=\edgewidth,dotted,red1] (x) to[out=-135, in=135] (y);
				\draw [line width=\edgewidth,dotted,red1] (x) to[out=-45, in=45] (y);
			\draw [line width=\edgewidth,dotted,green1] (x) -- (x') ;
				\draw [line width=\edgewidth,dotted,green1] (y) -- (y') ;
				\draw [line width=\edgewidth,dotted,green1] (x') to[out=-135, in=135] (y');
				\draw [line width=\edgewidth,dotted,green1] (x') to[out=-45, in=45] (y');

			\end{tikzpicture}
			\caption{Flat $(xy)$-paths and an $xy$-halo forming a subdivision of \cref{fig:graph8}.}
			\label{fig:halo_case_1}
			
		\end{subfigure}
			\hfill 
			\begin{subfigure}[t]{0.28\textwidth}
			\centering
			\begin{tikzpicture}
			\def\vtxSize{0.5cm}
				\def\width{1.25cm}
			\def\height{1.25cm}
				\def\edgewidth{1.5pt}
				\def\s{2.5pt}
                \def\angle1{63.4}
                \def\angle2{26.6}
			\node[draw,circle,minimum size=\vtxSize,inner sep=0pt] (x) at (0,\height) {$x$};
				\node[draw,circle,minimum size=\vtxSize,inner sep=0pt] (y) at (-\width,-\height) {$y$};

				\node[draw,circle,fill,inner sep=\s] (y') at  (\width, -\height){};

			\draw [line width=\edgewidth,dotted,red1] (y) to[out=83.4, in=-136.6] (x);
            \draw [line width=\edgewidth,dotted,red1] (y) to[out=43.4, in=-96.6] (x);
				\draw [line width=\edgewidth,dotted,green1] (y) -- (y') ;
				\draw [line width=\edgewidth,dotted,green1] (y') to[out=136.6, in=-83.4] (x);
				\draw [line width=\edgewidth,dotted,green1] (y') to[out=96.6, in=-43.4] (x);

			\end{tikzpicture}
			\caption{Flat $(xy)$-paths and an $xy$-halo forming a subdivision of \cref{fig:graph6}.}
			\label{fig:halo_case_2}
			
		\end{subfigure}
		\hfill
		\begin{subfigure}[t]{0.28\textwidth}
			\centering
			\begin{tikzpicture}
			\def\vtxSize{0.5cm}
				\def\width{1.25cm}
			\def\height{1.25cm}
				\def\edgewidth{1.5pt}
				\def\s{2.5pt}
                \def\angle1{63.4}
                \def\angle2{26.6}
			\node[draw,circle,minimum size=\vtxSize,inner sep=0pt] (x) at (0,\height) {$x$};
				\node[draw,circle,minimum size=\vtxSize,inner sep=0pt] (y) at (0,-\height) {$y$};

			\draw [line width=\edgewidth,dotted,red1] (y) to[out=150, in=-150] (x);
            \draw [line width=\edgewidth,dotted,red1] (y) to[out=120, in=-120] (x);

				\draw [line width=\edgewidth,dotted,green1] (y) to[out=60, in=-60] (x);
				\draw [line width=\edgewidth,dotted,green1] (y) to[out=30, in=-30] (x);

			\end{tikzpicture}
			\caption{Flat $(xy)$-paths and an $xy$-halo forming a subdivision of \cref{fig:graph1}.}
			\label{fig:halo_case_3}
			
		\end{subfigure}
		\caption{Parallel flat $(xy)$-paths (in red) and an $xy$-halo (in green) form a graph which is not s.m.}
        \label{fig:parrallel_plus_halo}
	\end{figure}
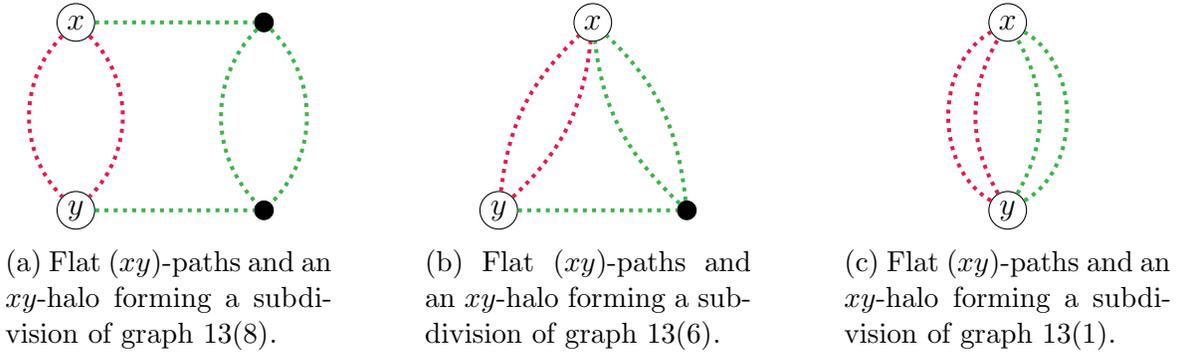
We are now ready to prove \cref{thm:minor_closed}.
\begin{proof}{[\Cref{thm:minor_closed}]}
	We already know that strictly metrizable graph are closed under topological minors, \cref{prop:top_minor}. So to prove this claim, it suffices to show that
    s.m.\ graphs are closed under edge contraction. Namely, that if $G$ is s.m.\ and $G' = G/xy$, then 
    $G'$ is also s.m. To this end, let us remove all compliant edges from $G'$ (we may need to
    recurse on this until we arrive at a graph with no compliant edges).
    By \cref{prop:compliant_can_be_deleted}, the resulting graph $G''$ is s.m.\ iff $G'$ is s.m. We will show that $G''$ is actually a topological minor of $G$ and
therefore s.m. This will allow us to deduce that $G'$ is s.m., as required.
	
	By \cref{thm:structure}, there are a handful of possible forms that $G$ can take. We will go through this list case by case. Prior to that, we make some reductions. First, we may clearly assume that $G$ is $2$-connected, otherwise we work with each $2$-connected component separately. Also, we may assume that $\deg(x),\deg(y)\geq 3$  for the contracted edge $xy$. In other words, we may assume both $x$ and $y$ are branch vertices. Indeed, if $\min(\deg(x),\deg(y))=2$ then $G'$ is just a topological minor of $G$, hence s.m. 
    We may also assume that no edge in $G$ is compliant. To prove this, let us first assume that $e\neq xy$  is a compliant edge. Since $\deg(x),\deg(y)\geq 3$, it is not difficult to see that $e$ is also compliant in $G'$, and so we may freely remove $e$. If $xy$ is compliant, then contracting it yields a non-$2$-connected graph. Indeed, the contraction of $xy$ turns a flat $xy$-path into a $2$-connected component. Therefore, we may delete these $xy$-paths in $G$ so that $xy$ is no longer compliant and then contract $xy$.\\
     \begin{figure}[h]
			\centering
			\begin{subfigure}[t]{0.28\textwidth}
				\centering
				\begin{tikzpicture}
				\def\vtxSize{0.4cm}
				\def\r{1.75cm}
				\def\edgewidth{1.5pt}
				\node[draw,circle,minimum size=\vtxSize,inner sep=1pt]  (1) at (0*360/3 +90: \r) {\small$v_1$};
				\node[draw,circle,minimum size=\vtxSize,inner sep=1pt]  (2) at (1*360/3 +90: \r) {\small$v_2$};
				\node[draw,circle,minimum size=\vtxSize,inner sep=1pt]  (3) at (2*360/3 +90: \r) {\small$v_3$};
				\node[draw,circle,minimum size=\vtxSize,inner sep=1pt]  (4) at (0,0) {\small$v_4$};
				
				\draw [line width=\edgewidth,dotted] (1) -- (2) ;
				\draw [line width=\edgewidth] (1) -- (3) ;
                    \draw [line width=\edgewidth,dotted] (1) to[out=-10,in=70] (3) ;
				\draw [line width=\edgewidth,dotted] (1) -- (4) ;
				\draw [line width=\edgewidth,dotted] (2) -- (3) ;
				\draw [line width=\edgewidth,dotted] (2) -- (4) ;
				\draw [line width=\edgewidth,dotted] (3) -- (4) ;

				\end{tikzpicture}
				\caption{A graph with a compliant edge $v_1v_3$.}
				\label{fig:contracting_compliant_1}
				
			\end{subfigure}
				\hspace{30mm}
				\begin{subfigure}[t]{0.28\textwidth}
				\centering
				\begin{tikzpicture}
					\def\vtxSize{0.4cm}
				\def\r{1.75cm}
				\def\edgewidth{1.5pt}
                    \def\rad{.75cm}
                   
                    \coordinate (u) at (\r,0);
				
				\node[draw,circle,minimum size=\vtxSize,inner sep=1pt]  (2) at (-\r,0) {\small$v_2$};
				\node[draw,circle,minimum size=\vtxSize,inner sep=1pt]  (4) at (0,0) {\small$v_4$};

				\draw [line width=\edgewidth,dotted] (u) to[out=150,in=30] (4) ;
                    \draw [line width=\edgewidth,dotted] (u) to[out=-150,in=-30] (4) ;
                    \draw [line width=\edgewidth,dotted] (u) to[out=120,in=60] (2) ;
                    \draw [line width=\edgewidth,dotted] (u) to[out=-120,in=-60] (2) ;
                    \draw [line width=\edgewidth,dotted]  ($(u)+(\rad,0)$) circle (\rad);
				
				\draw [line width=\edgewidth,dotted] (2) -- (4)  ;
				 \node[draw,circle,minimum size=\vtxSize,inner sep=1pt, fill=white]  at (\r,0) {\small$u$};
				\end{tikzpicture}
				\caption{Contracting a compliant edge yields a graph with a cut vertex.}
				\label{fig:subdivided_K4_contraction}
				
			\end{subfigure}
			\caption{Contracting a compliant edge yields a graph which is no longer $2$-connected.}
		\end{figure}
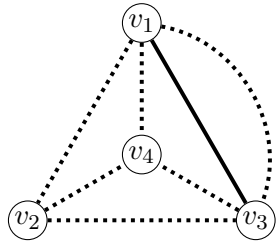
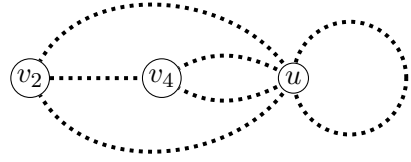
	By \cref{thm:structure} we may assume that $G$ is either $K_5$, $W_5$ or a subdivision of $K_{2,3}$, $K_4$, $W_4$ or $W_4$'. We go over each of these cases.
	\begin{itemize}
		\item $G$ is a subdivision of $K_{2,3}$. In this case, there is
        no edge $xy$ in $G$ for which $\deg(x),\deg(y)\geq 3$, and so there is nothing to show.
		\item $G$ is a subdivision of $K_4$, \cref{fig:subdivided_K4}. $G$ has $4$ branch vertices called $v_1,v_2,v_3, v_4$. Let $P_{i,j}$ denote the path between vertex $v_i$ and $v_j$. Say that $P_{1,4}=v_1v_4$ is the edge we contract. We claim that at least one of the four paths $P_{1,2}$, $P_{1,3}$, $P_{2,4}$ and $P_{3,4}$ is also edge, for otherwise we obtain a subdivision of \cref{fig:graph7} and $G$ is not s.m. So w.l.o.g. $P_{3,4}=v_3v_4$. We now contract the edge $v_1v_4$ and call this new vertex $u$, \cref{fig:subdivided_K4_contraction}. After contracting $v_1v_4$, the edge $v_3v_4$ is compliant. After we delete it, we obtain a topological minor of our original graph $G$, \cref{fig:del_comp_edge_K4}.
		 \begin{figure}[h]
			\centering
			\begin{subfigure}[t]{0.28\textwidth}
				\centering
				\begin{tikzpicture}
				\def\vtxSize{0.4cm}
				\def\r{1.75cm}
				\def\edgewidth{1.5pt}
				\node[draw,circle,minimum size=\vtxSize,inner sep=1pt]  (1) at (0*360/3 +90: \r) {\small$v_1$};
				\node[draw,circle,minimum size=\vtxSize,inner sep=1pt]  (2) at (1*360/3 +90: \r) {\small$v_2$};
				\node[draw,circle,minimum size=\vtxSize,inner sep=1pt]  (3) at (2*360/3 +90: \r) {\small$v_3$};
				\node[draw,circle,minimum size=\vtxSize,inner sep=1pt]  (4) at (0,0) {\small$v_4$};
				
				\draw [line width=\edgewidth,dotted] (1) -- (2) ;
				\draw [line width=\edgewidth,dotted] (1) -- (3) ;
				\draw [line width=\edgewidth] (1) -- (4) ;
				\draw [line width=\edgewidth,dotted] (2) -- (3) ;
				\draw [line width=\edgewidth,dotted] (2) -- (4) ;
				\draw [line width=\edgewidth,dotted] (3) -- (4) ;

				\end{tikzpicture}
				\caption{A subdivided $K_4$.}
				\label{fig:subdivided_K4}
				
			\end{subfigure}
				\hfill 
				\begin{subfigure}[t]{0.28\textwidth}
				\centering
				\begin{tikzpicture}
					\def\vtxSize{0.4cm}
				\def\r{1.75cm}
				\def\edgewidth{1.5pt}
				\node[draw,circle,minimum size=\vtxSize,inner sep=1pt]  (1) at (0*360/3 +90: \r) {\small$u$};
				\node[draw,circle,minimum size=\vtxSize,inner sep=1pt]  (2) at (1*360/3 +90: \r) {\small$v_2$};
				\node[draw,circle,minimum size=\vtxSize,inner sep=1pt]  (3) at (2*360/3 +90: \r) {\small$v_3$};

				\draw [line width=\edgewidth,dotted] (1) -- (2) ;
				\draw [line width=\edgewidth] (1) -- (3) ;
				
				\draw [line width=\edgewidth,dotted] (2) -- (3) ;
				\draw [line width=\edgewidth,dotted] (1) to[out=170, in=120] (2);
				\draw [line width=\edgewidth,dotted] (1) to[out=10, in=60] (3);
				
				\end{tikzpicture}
				\caption{Contracting an edge of a subdivided $K_4$.}
				\label{fig:subdivided_K4_contraction}
				
			\end{subfigure}
			\hfill
			\begin{subfigure}[t]{0.28\textwidth}
				\centering
				\begin{tikzpicture}
					\def\vtxSize{0.4cm}
					\def\r{1.75cm}
					\def\edgewidth{1.5pt}
					\node[draw,circle,minimum size=\vtxSize,inner sep=1pt]  (1) at (0*360/3 +90: \r) {\small$u$};
					\node[draw,circle,minimum size=\vtxSize,inner sep=1pt]  (2) at (1*360/3 +90: \r) {\small$v_2$};
					\node[draw,circle,minimum size=\vtxSize,inner sep=1pt]  (3) at (2*360/3 +90: \r) {\small$v_3$};

					\draw [line width=\edgewidth,dotted] (1) -- (2) ;

					\draw [line width=\edgewidth,dotted] (2) -- (3) ;
					\draw [line width=\edgewidth,dotted] (1) to[out=170, in=120] (2);
					\draw [line width=\edgewidth,dotted] (1) to[out=10, in=60] (3);
					
					\end{tikzpicture}
	
				\caption{Deleting a compliant edge in a contracted subdivided $K_4$.}
				\label{fig:del_comp_edge_K4}
				
			\end{subfigure}
			\caption{Contracting an edge in a subdivided s.m. $K_4$ yields an s.m. graph. }
		\end{figure}
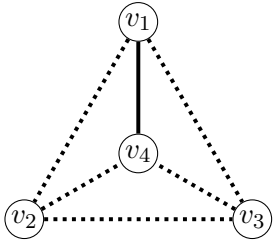
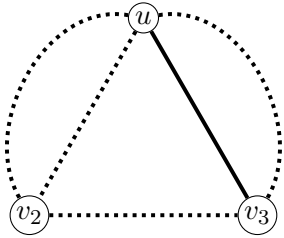
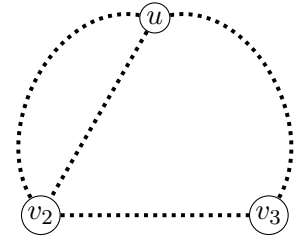
		\item $G$ is a subdivision of $W_4$. Let $v_0,v_1,v_2,v_3,v_4$ be the branch vertices of $G$, where $v_0$ is the `central' vertex and $v_1,v_2,v_3,v_4$ appear in this order along the outer cycle, \cref{fig:subdivided_W4}. As before, let $P_{i,j}$ be the path between $v_{i}$ and $v_{j}$.\\
		We claim that if $P_{i,j}$ and $P_{k,l}$ share an endpoint then at least one of them is an edge. Up to symmetries there are four cases to consider. Case (i): If neither $P_{0,1}$ nor $P_{0,4}$ are edges, then $G$ contains a subdivision of \cref{fig:graph6}, \cref{fig:Incident_Path_case_1}. Case (ii): If neither $P_{0,1}$ nor $P_{1,4}$ are edges, then $G$ contains a subdivision of \cref{fig:graph5}, \cref{fig:Incident_Path_case_2}. Case (iii): If neither $P_{1,4}$ nor $P_{3,4}$ are edges, then $G$ contains a subdivision of \cref{fig:graph4}, \cref{fig:Incident_Path_case_3}. Case (iv): If neither $P_{0,1}$ nor $P_{0,3}$ are edges, then $G$ again contains a subdivision of \cref{fig:graph4}.\\
		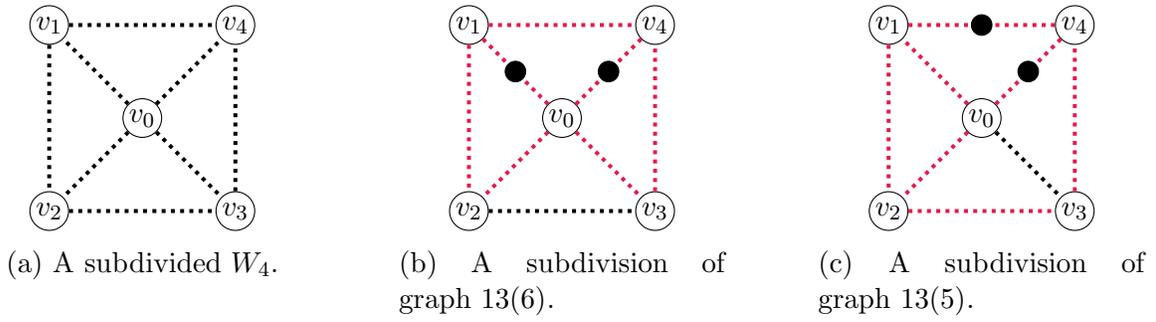
\begin{figure}[h]
			\centering
			\begin{subfigure}[t]{0.28\textwidth}
				\centering
				\begin{tikzpicture}
				\def\vtxSize{0.4cm}
				\def\r{1.75cm}
				\def\edgewidth{1.5pt}
				\node[draw,circle,minimum size=\vtxSize,inner sep=1pt]  (0) at (0,0) {\small$v_0$};
				\node[draw,circle,minimum size=\vtxSize,inner sep=1pt]  (1) at (0*360/4 +135: \r) {\small$v_1$};
				\node[draw,circle,minimum size=\vtxSize,inner sep=1pt]  (2) at (1*360/4 +135: \r) {\small$v_2$};
				\node[draw,circle,minimum size=\vtxSize,inner sep=1pt]  (3) at (2*360/4 +135: \r) {\small$v_3$};
				\node[draw,circle,minimum size=\vtxSize,inner sep=1pt]  (4) at (3*360/4 +135: \r) {\small$v_4$};

				\draw [line width=\edgewidth,dotted] (0) -- (1) ;
				\draw [line width=\edgewidth,dotted] (0) -- (2) ;
				\draw [line width=\edgewidth,dotted] (0) -- (3) ;
				\draw [line width=\edgewidth,dotted] (0) -- (4) ;
				\draw [line width=\edgewidth,dotted] (1) -- (2) ;
				\draw [line width=\edgewidth,dotted] (1) -- (4) ;
				\draw [line width=\edgewidth,dotted] (2) -- (3) ;
				\draw [line width=\edgewidth,dotted] (3) -- (4) ;

				\end{tikzpicture}
				\caption{A subdivided $W_4$.}
				\label{fig:subdivided_W4}
				
			\end{subfigure}
				\hfill 
				\begin{subfigure}[t]{0.28\textwidth}
				\centering
				\begin{tikzpicture}
					\def\vtxSize{0.4cm}
					\def\r{1.75cm}
					\def\edgewidth{1.5pt}
					\node[draw,circle,minimum size=\vtxSize,inner sep=1pt]  (0) at (0,0) {\small$v_0$};
					\node[draw,circle,minimum size=\vtxSize,inner sep=1pt]  (1) at (0*360/4 +135: \r) {\small$v_1$};
					\node[draw,circle,minimum size=\vtxSize,inner sep=1pt]  (2) at (1*360/4 +135: \r) {\small$v_2$};
					\node[draw,circle,minimum size=\vtxSize,inner sep=1pt]  (3) at (2*360/4 +135: \r) {\small$v_3$};
					\node[draw,circle,minimum size=\vtxSize,inner sep=1pt]  (4) at (3*360/4 +135: \r) {\small$v_4$};
					\node[draw,circle, fill,scale=0.7] (x1) at ($(0)!0.5!(1)$) {};
					\node[draw,circle, fill,scale=0.7] (x2) at ($(0)!0.5!(4)$) {};

					\draw [line width=\edgewidth,dotted,red1] (0) -- (x1) ;
					\draw [line width=\edgewidth,dotted,red1] (1) -- (x1) ;
					\draw [line width=\edgewidth,dotted,red1] (0) -- (2) ;
					\draw [line width=\edgewidth,dotted,red1] (0) -- (3) ;
					\draw [line width=\edgewidth,dotted,red1] (0) -- (x2) ;
					\draw [line width=\edgewidth,dotted,red1] (4) -- (x2) ;
					\draw [line width=\edgewidth,dotted,red1] (1) -- (2) ;
					\draw [line width=\edgewidth,dotted,red1] (1) -- (4) ;
					\draw [line width=\edgewidth,dotted] (2) -- (3) ;
					\draw [line width=\edgewidth,dotted,red1] (3) -- (4) ;

				\end{tikzpicture}
				\caption{A subdivision of \cref{fig:graph6}.}
				\label{fig:Incident_Path_case_1}
				
			\end{subfigure}
			\hfill
			\begin{subfigure}[t]{0.28\textwidth}
				\centering
				\begin{tikzpicture}
					\def\vtxSize{0.4cm}
					\def\r{1.75cm}
					\def\edgewidth{1.5pt}
					\node[draw,circle,minimum size=\vtxSize,inner sep=1pt]  (0) at (0,0) {\small$v_0$};
					\node[draw,circle,minimum size=\vtxSize,inner sep=1pt]  (1) at (0*360/4 +135: \r) {\small$v_1$};
					\node[draw,circle,minimum size=\vtxSize,inner sep=1pt]  (2) at (1*360/4 +135: \r) {\small$v_2$};
					\node[draw,circle,minimum size=\vtxSize,inner sep=1pt]  (3) at (2*360/4 +135: \r) {\small$v_3$};
					\node[draw,circle,minimum size=\vtxSize,inner sep=1pt]  (4) at (3*360/4 +135: \r) {\small$v_4$};
					\node[draw,circle, fill,scale=0.7] (x1) at ($(1)!0.5!(4)$) {};
					\node[draw,circle, fill,scale=0.7] (x2) at ($(0)!0.5!(4)$) {};

					\draw [line width=\edgewidth,dotted,red1] (0) -- (1) ;
					\draw [line width=\edgewidth,dotted,red1] (0) -- (2) ;
					\draw [line width=\edgewidth,dotted] (0) -- (3) ;
					\draw [line width=\edgewidth,dotted,red1] (0) -- (x2) ;
					\draw [line width=\edgewidth,dotted,red1] (4) -- (x2) ;
					\draw [line width=\edgewidth,dotted,red1] (1) -- (2) ;
					\draw [line width=\edgewidth,dotted,red1] (1) -- (x1) ;
					\draw [line width=\edgewidth,dotted,red1] (4) -- (x1) ;
					\draw [line width=\edgewidth,dotted,red1] (2) -- (3) ;
					\draw [line width=\edgewidth,dotted,red1] (3) -- (4) ;

				\end{tikzpicture}
	
				\caption{A subdivision of \cref{fig:graph5}.}
				\label{fig:Incident_Path_case_2}
				
			\end{subfigure}
			\caption{A subdivided $W_4$ containing adjacent subdivided edges is not s.m.}
		\end{figure}
		We now consider what happens when we contract an edge in $G$. Up to symmetries there are two cases to consider. First suppose that we contract the edge $P_{1,4}=v_1v_4$ to a vertex $u$. As mentioned above, at least one of $P_{0,1}$ and $P_{0,4}$ is an edge. After contracting $v_1v_4$, $P_{0,1}$ and $P_{0,4}$ become parallel $(v_0u)$-paths. At least one these is a compliant edge in $G/v_1v_4$. Deleting this edge we obtain a subdivided $K_4$ which is a topological minor of $G$, \cref{fig:subdivided_W4_contraction_1}. \\
		Next suppose we contract the edge $P_{0,4}=v_0v_4$ to a vertex $u$. At least one of $P_{0,3}$, $P_{3,4}$ and $P_{0,1}$, $P_{1,4}$ are edges. Therefore, either $P_{0,3}$ or $P_{3,4}$ (similarly, $P_{0,1}$ or $P_{1,4}$) becomes a compliant edge after contracting $v_0v_4$. Removing these compliant edges yields a topological minor of $G$, \cref{fig:subdivided_W4_contraction_2}.\\
		\begin{figure}[h]
			\centering
			\begin{subfigure}[t]{0.28\textwidth}
				\centering
				\begin{tikzpicture}
					\def\vtxSize{0.4cm}
					\def\r{1.75cm}
					\def\edgewidth{1.5pt}
					\node[draw,circle,minimum size=\vtxSize,inner sep=1pt]  (0) at (0,0) {\small$v_0$};
					\node[draw,circle,minimum size=\vtxSize,inner sep=1pt]  (1) at (0*360/4 +135: \r) {\small$v_1$};
					\node[draw,circle,minimum size=\vtxSize,inner sep=1pt]  (2) at (1*360/4 +135: \r) {\small$v_2$};
					\node[draw,circle,minimum size=\vtxSize,inner sep=1pt]  (3) at (2*360/4 +135: \r) {\small$v_3$};
					\node[draw,circle,minimum size=\vtxSize,inner sep=1pt]  (4) at (3*360/4 +135: \r) {\small$v_4$};
					\node[draw,circle, fill,scale=0.7] (x1) at ($(1)!0.5!(4)$) {};
					\node[draw,circle, fill,scale=0.7] (x2) at ($(3)!0.5!(4)$) {};

					\draw [line width=\edgewidth,dotted,red1] (0) -- (1) ;
					\draw [line width=\edgewidth,dotted] (0) -- (2) ;
					\draw [line width=\edgewidth,dotted,red1] (0) -- (3) ;
					\draw [line width=\edgewidth,dotted,red1] (0) -- (4) ;

					\draw [line width=\edgewidth,dotted,red1] (1) -- (2) ;
					\draw [line width=\edgewidth,dotted,red1] (1) -- (x1) ;
					\draw [line width=\edgewidth,dotted,red1] (4) -- (x1) ;
					\draw [line width=\edgewidth,dotted,red1] (2) -- (3) ;
					\draw [line width=\edgewidth,dotted,red1] (3) -- (x2) ;
					\draw [line width=\edgewidth,dotted,red1] (4) -- (x2) ;

				\end{tikzpicture}
	
				\caption{A subdivision of \cref{fig:graph4}.}
				\label{fig:Incident_Path_case_3}
				
			\end{subfigure}
				\hfill 
				\begin{subfigure}[t]{0.28\textwidth}
				\centering
				\begin{tikzpicture}
					\def\vtxSize{0.4cm}
				\def\r{1.75cm}
				\def\edgewidth{1.5pt}
				\node[draw,circle,minimum size=\vtxSize,inner sep=1pt]  (0) at (0,0) {\small$v_0$};
				\node[draw,circle,minimum size=\vtxSize,inner sep=1pt]  (u) at (0*360/3 +90: \r) {\small$u$};
				\node[draw,circle,minimum size=\vtxSize,inner sep=1pt]  (2) at (1*360/3 +90: \r) {\small$v_2$};
				\node[draw,circle,minimum size=\vtxSize,inner sep=1pt]  (3) at (2*360/3 +90: \r) {\small$v_3$};

				\draw [line width=\edgewidth,dotted] (0) -- (2) ;
				\draw [line width=\edgewidth,dotted] (0) -- (3) ;
				
				\draw [line width=\edgewidth,dotted] (2) -- (3) ;
				\draw [line width=\edgewidth,dotted] (u) to[out=170, in=120] (2);
				\draw [line width=\edgewidth,dotted] (u) to[out=10, in=60] (3);

				\draw [line width=\edgewidth,dotted] (u) to[out=225, in=135] (0);
				\draw [line width=\edgewidth] (u) to[out=-45, in=45] (0);
				
				\end{tikzpicture}
				\caption{Contracting an edge of a subdivided $W_4$.}
				\label{fig:subdivided_W4_contraction_1}
				
			\end{subfigure}
			\hfill
			\begin{subfigure}[t]{0.28\textwidth}
				\centering
				\begin{tikzpicture}
					\def\vtxSize{0.4cm}
					\def\r{1.75cm}
					\def\edgewidth{1.5pt}
				
				\node[draw,circle,minimum size=\vtxSize,inner sep=1pt]  (1) at (0*360/4 +135: \r) {\small$v_1$};
				\node[draw,circle,minimum size=\vtxSize,inner sep=1pt]  (2) at (1*360/4 +135: \r) {\small$v_2$};
				\node[draw,circle,minimum size=\vtxSize,inner sep=1pt]  (3) at (2*360/4 +135: \r) {\small$v_3$};
				\node[draw,circle,minimum size=\vtxSize,inner sep=1pt]  (u) at (3*360/4 +135: \r) {\small$u$};

				\draw [line width=\edgewidth,dotted] (1) -- (2) ;
				\draw [line width=\edgewidth,dotted] (2) -- (3) ;
				\draw [line width=\edgewidth,dotted] (u) to[out=-70, in=70] (3);
				\draw [line width=\edgewidth] (u) to[out=-110, in=110] (3);
				\draw [line width=\edgewidth,dotted] (u) to[out=160, in=20] (1);
				\draw [line width=\edgewidth] (u) to[out=-160, in=-20] (1);
				\draw [line width=\edgewidth,dotted] (u) -- (2) ;

				\end{tikzpicture}
	
				\caption{Contracting a spoke in a subdivided $W_4$.}
				\label{fig:subdivided_W4_contraction_2}
				
			\end{subfigure}
			\caption{Contracting an edge in a subdivided s.m. $W_4$ yields an s.m. graph.}
		\end{figure}
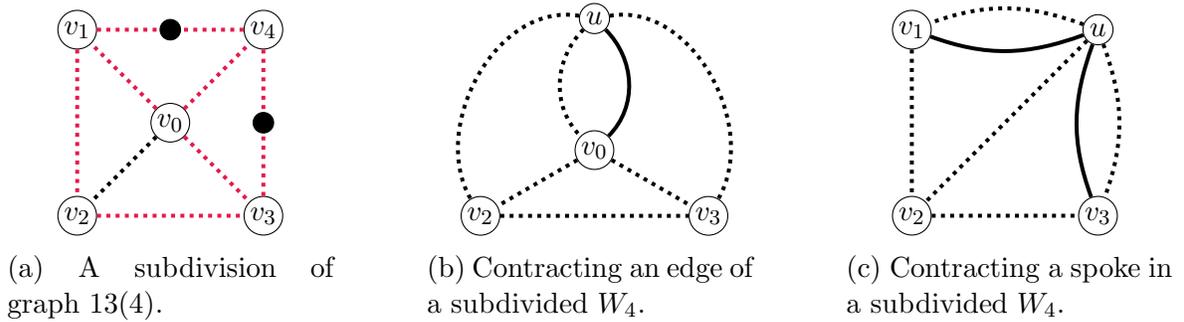
		\item $G$ is a subdivision of $W_4'$. As before, we label $G$'s branch vertices by $v_0,v_1,v_2,v_3,v_4$, \cref{fig:subdivided_W4'}. Similarly, we let $P_{i,j}$ be the flat path between $v_{i}$ and $v_{j}$. The only difference is that now we also have a flat path $P_{1,3}$ between $v_1$ and $v_3$. But if $P_{1,3}$ has at least $2$ edges, then $G$ contains a subdivision of $K_{2,4}$, \cref{fig:graph1}. Consequently, we may assume that $P_{1,3}$ is an edge. Moreover, as argued in the previous case, if $P_{i,j}$ and $P_{k,l}$ share an endpoint then at least one of them is an edge. So, let us see what happens when we contract an edge $e$ in this graph. Up to symmetries there are three case to consider: (i) $e =v_1v_4$ (ii) $e=v_0v_1$ (iii) $e=v_1v_3$. The proofs for Cases (i) and (ii) are nearly identical to the previous case of $W_4$, since ignoring the edge $v_1v_3$, the graph $G$ is a subdivision of $W_4$. For case (iii), we contract $v_1v_3$ to a vertex $u$ and observe that for $i=0,2,4$, $P_{i,1}$, $P_{i,3}$ become parallel $(v_iu)$-paths, at least one of which is a compliant edge, \cref{fig:subdivided_W4'_contraction_1}. Deleting these compliant edges, we again get a graph which is a topological minor of $G$.
		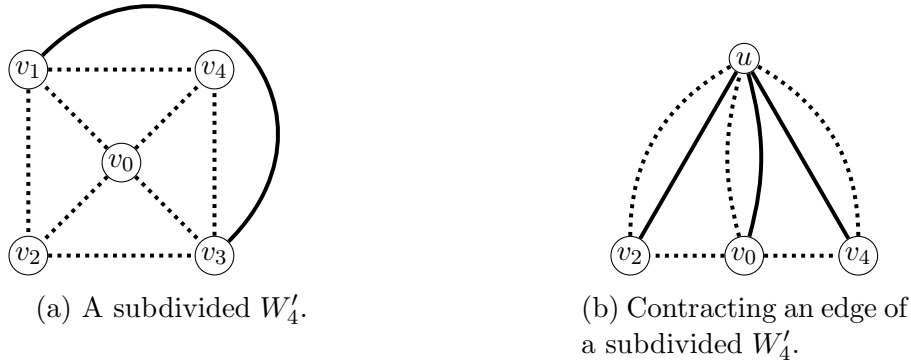
\begin{figure}[h]
			\centering
			\begin{subfigure}[t]{0.28\textwidth}
				\centering
				\begin{tikzpicture}
				\def\vtxSize{0.4cm}
				\def\r{1.75cm}
				\def\edgewidth{1.5pt}
				\node[draw,circle,minimum size=\vtxSize,inner sep=1pt]  (0) at (0,0) {\small$v_0$};
				\node[draw,circle,minimum size=\vtxSize,inner sep=1pt]  (1) at (0*360/4 +135: \r) {\small$v_1$};
				\node[draw,circle,minimum size=\vtxSize,inner sep=1pt]  (2) at (1*360/4 +135: \r) {\small$v_2$};
				\node[draw,circle,minimum size=\vtxSize,inner sep=1pt]  (3) at (2*360/4 +135: \r) {\small$v_3$};
				\node[draw,circle,minimum size=\vtxSize,inner sep=1pt]  (4) at (3*360/4 +135: \r) {\small$v_4$};

				\draw [line width=\edgewidth,dotted] (0) -- (1) ;
				\draw [line width=\edgewidth,dotted] (0) -- (2) ;
				\draw [line width=\edgewidth,dotted] (0) -- (3) ;
				\draw [line width=\edgewidth,dotted] (0) -- (4) ;
				\draw [line width=\edgewidth,dotted] (1) -- (2) ;
				\draw [line width=\edgewidth,dotted] (1) -- (4) ;
				\draw [line width=\edgewidth,dotted] (2) -- (3) ;
				\draw [line width=\edgewidth,dotted] (3) -- (4) ;
				\draw [line width=\edgewidth] (1) to[out=45, in=45, distance=1.5* \r ] (3);

				\end{tikzpicture}
				\caption{A subdivided $W_4'$.}
				\label{fig:subdivided_W4'}
				
			\end{subfigure}
				\hspace{30mm}
				\begin{subfigure}[t]{0.28\textwidth}
				\centering
				\begin{tikzpicture}
					\def\vtxSize{0.4cm}
					\def\r{1.75cm}
					\def\edgewidth{1.5pt}
					
					\node[draw,circle,minimum size=\vtxSize,inner sep=1pt]  (u) at (0*360/3 +90: \r) {\small$u$};
					\node[draw,circle,minimum size=\vtxSize,inner sep=1pt]  (2) at (1*360/3 +90: \r) {\small$v_2$};
					\node[draw,circle,minimum size=\vtxSize,inner sep=1pt]  (4) at (2*360/3 +90: \r) {\small$v_4$};
					\node[draw,circle,minimum size=\vtxSize,inner sep=1pt]  (0) at ($(2)!0.5!(4)$) {\small$v_0$};

					\draw [line width=\edgewidth,dotted] (0) -- (2) ;
					\draw [line width=\edgewidth,dotted] (0) -- (4) ;
					\draw [line width=\edgewidth,dotted] (2) to[out=90, in=-150] (u);
					\draw [line width=\edgewidth] (2) to(u);
					\draw [line width=\edgewidth,dotted] (0) to[out=105, in=-105] (u);
					\draw [line width=\edgewidth] (0) to[out=75, in=-75] (u);
					\draw [line width=\edgewidth,dotted] (4) to[out=90, in=-30] (u);
					\draw [line width=\edgewidth] (4) to(u);

				\end{tikzpicture}
				\caption{Contracting an edge of a subdivided $W_4'$.}
				\label{fig:subdivided_W4'_contraction_1}
				
			\end{subfigure}
			\caption{Contracting an edge in a subdivided s.m. $W_4'$ yields an s.m. graph.}
		\end{figure}
		\item In the remaining case $G$ is either an (honest) $K_5$ or $W_5$. If $G=K_5$, then contracting an edge yields $K_4$, clearly a topological minor of $K_5$. When $G=W_5$ there  are two possible edge contractions. One yields $W_4$, a topological minor of $W_5$. The other generates an outerplanar graph, which becomes a $5$-cycle when compliant edges get deleted.
	\end{itemize}
\end{proof}

\section{Persistent vs.\ Zero-Weight Edges}\label{sec:persistent_edges}

Suppose that $\cal P$ is a consistent path system in $G=(V,E)$ which can be realized as the unique shortest path with respect to some edge weights. 
What can be said about such edge weights? Specifically,
given an edge $e\in E$, can $\cal P$ be induced by some edge weights 
$w:E\to \mathbb{R}_{\ge0}$ with $w(e)=0$? 
We give a necessary condition for this and prove, 
that this condition is also sufficient for strictly metrizable graphs. 

Other questions about functions $w:E\to \mathbb{R}_{>0}$ that induces $\cal P$ 
suggest themselves. E.g., \cite{BBW} considers the least possible aspect ratio of such $w$.

\subsection{Persistence}
Let us recall the notion of a {\em persistent edge} from \cite{CL}. Let $\mathcal{P}$ be a consistent path system of a graph $G=(V,E)$. An edge $uv\in E$ is said to be {\em persistent} (w.r.t.\ $\mathcal{P}$) if for every vertex $x\in V$, either $P_{x,v} = P_{x,u}uv$ or $P_{x,u} = P_{x,v}vu$. In other words, either $x$ goes to $v$ via $u$ or goes to $u$ via $v$. Intuitively, an edge $uv$ persistent if it is `infinitesimally short' w.r.t.\ $\mathcal{P}$. 

Let $\mathcal{P}$ be a consistent path system of a graph $G=(V,E)$. We note the following:
\begin{itemize}
\item 
For an edge $e\in E$,
the path system obtained from $\mathcal{P}$ upon contracting $e$ is consistent iff $e$ is persistent. \\
Indeed, if $e=uv$ is not persistent then there exists an vertex $x\in V$ such that $P_{x,u}$ and $P_{x,v}$ are disjoint except at $x$. Contracting $uv$ to a vertex $z$ yields two different $(xz)$-paths, namely $P_{x,u}/e$ and $P_{x,v}/e$. \\
On the other hand, say we contract the persistent edge $e=uv$ to a vertex $z$. The
persistence of $e$ means that
for all $x$ either $v\in P_{x,u}$ or $u\in P_{x,v}$. Consequently, the $xz$-path
in the $\mathcal{P}/e$ is well defined. Namely, it is $P_{x,u}/e=P_{x,v}/e$. 
The consistency of $\mathcal{P}/e$ follows from the consistency of $\mathcal{P}$.
\item 
As observed in \cite{CL}, given a vertex $x$,
the collection of all $xy$-paths in $\mathcal{P}$ form a tree $T_x$ rooted at $x$. 
The edge $e\in E$ is persistent w.r.t.\ $\mathcal{P}$ iff it belongs to $E(T_x)$ for every vertex $x$.
\end{itemize}

We wish to return to the intuition that
persistent edges are 'infinitesimally short', and make it concrete.
Up until now, metrizability was defined in terms of positive edge weights. So we first need to extend the notion of 
metrizability to non-negative weight 
functions.
 Let us say that a non-negative weight function $w:E\to \mathbb{R}_{\ge 0}$ {\em simply induces} a path system $\mathcal{P}$ on $G=(V,E)$ if for each $u,v\in V$,  $P_{u,v} \in \mathcal{P}$ is the {\em unique} shortest {\em simple} $(uv)$-path. 

If $w$ simply induces a path system $\mathcal{P}$ and $w$ is strictly positive,
then $w$ strictly induces $\mathcal{P}$ in the usual sense.
But if a non-negative $w$ simply induces a path system $\mathcal{P}$ then, by slightly increasing all edge weights, we clearly obtain a positive weight function strictly inducing $\mathcal{P}$. The following proposition sheds some light on the relation between zero-weight edges and persistent edges:

\begin{proposition}\label{lem:pers_vs_0}
\begin{enumerate}
\item 
Let $\mathcal{P}$ be a path system on $G=(V,E)$ simply induced by a weight function
$w:E\to \mathbb{R}_{\geq 0}$. If $w(e) = 0$ for some edge $e\in E$ then $e$ is $\mathcal{P}$-persistent.
\item
However, there exist strictly metrizable paths systems with persistent edges whose weight cannot be set to zero.
\item 
Let $\mathcal{P}$ be a consistent path system on a strictly metrizable graph.
Then there exists $w:E\to \mathbb{R}_{\geq 0}$ simply inducing $\mathcal{P}$ such that 
$w(e)=0$ for every $\mathcal{P}$-persistent edge $e\in E$.
\end{enumerate}
\end{proposition}
\begin{proof}
{\bf Part 1:}
If $e=uv$ is not persistent, then there exists $x\in V$ such that $P_{x,u}$ and $P_{x,v}$ have only $x$ in common. Let $Q_{x,u} = P_{x,v}vu$ be the $(xu)$-path obtained by concatenating $P_{x,v}$ and $vu$. Similarly, we set $Q_{x,v} = P_{x,u}uv$. Since $w$ induces $\mathcal{P}$,
	\[w(P_{x,v}) < w(Q_{x,v}) = w(P_{x,u}) +w(e) = w(P_{x,u}) .\]
	Similarly,
	\[w(P_{x,u}) < w(Q_{x,u}) = w(P_{x,v}) +w(e) = w(P_{x,v}) .\]
	Adding the two inequalities together we get $0<0$, a contradiction. \\

{\bf Part 2:} Consider the weighted graph in \cref{fig:persistent_non_zero}. These weights give rise to a path system $\mathcal{P}$, consisting of the unique shortest paths, where the edge $(3,4)$ is persistent.
\begin{figure}[h]
	\centering
	\begin{tikzpicture}[scale=0.8, every node/.style={scale=0.8}]
		\def\r{2.5cm}
		\def\l{3pt}
		\def\i{3pt}

		\node[draw,circle,inner sep=\i] (1) at (0*360/5 +90: \r) {$1$};
			\node[draw,circle,inner sep=\i] (2) at (1*360/5 +90: \r) {$2$};
			\node[draw,circle,inner sep=\i] (3) at (2*360/5 +90: \r) {$3$};
			\node[draw,circle,inner sep=\i] (4) at (3*360/5 +90: \r) {$4$};
			\node[draw,circle,inner sep=\i] (5) at (4*360/5 +90:\r) {$5$};
			
			\node[draw,circle,inner sep=\i] (6) at ($(1)!0.5!(3)$) {$6$};
			\node[draw,circle,inner sep=\i] (7) at  ($(1)!0.5!(4)$) {$7$};

			\draw [line width=\l,-] (1) -- (2) node[midway,label={[label distance = -10pt]135:\Large$5$}]{};
			\draw [line width=\l,-] (1) -- (5) node[midway,label={[label distance = -10pt]45:\Large$6$}]{};
			\draw [line width=\l,-] (1) -- (6) node[midway,label={[label distance = -5pt]180:\Large$4$}]{};
			\draw [line width=\l,-] (1) -- (7) node[midway,label={[label distance = -5pt]0:\Large$8$}]{};
			\draw [line width=\l,-] (2) -- (3) node[midway,label={[label distance = -3pt]-180:\Large$2$}]{};
			\draw [line width=\l,-] (3) -- (4) node[midway,label={[label distance = -5pt]-90:\Large$4$}]{};
			\draw [line width=\l,-] (3) -- (6) node[midway,label={[label distance = -11pt]135:\Large$5$}]{};
			\draw [line width=\l,-] (4) -- (5) node[midway,label={[label distance = -3pt]0:\Large$7$}]{};
			\draw [line width=\l,-] (4) -- (7) node[midway,label={[label distance = -11pt]45:\Large$5$}]{};

	\end{tikzpicture}
	\caption{The edge weights in the above graph strictly induce a path system where the edge $(3,4)$ is persistent.}
	\label{fig:persistent_non_zero}
\end{figure}
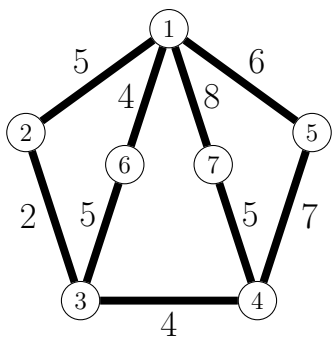
Notice the paths $(2,1,5)$, $(2,3,6)$, $(5,4,7)$, and $(6,1,7)$ are in this path system. Therefore, any weight function $w:E\to \mathbb{R}_{\geq 0}$ simply inducing this system must satisfy 
\begin{equation*}
	\begin{split}
		w_{1,2}+w_{1,5} & < w_{2,3}+w_{3,4}+w_{4,5}\\
		w_{2,3}+w_{3,6} & < w_{1,2}+w_{1,6}\\
		w_{4,5}+w_{4,7} & < w_{1,5}+w_{1,7}\\
		w_{1,6}+w_{1,7} & < w_{3,6}+w_{3,4}+w_{4,7}.
	\end{split}
\end{equation*}
Adding these inequalities together yields $0<2w_{3,4}$, implying that under any weights simply 
inducing $\mathcal{P}$ the persistent edge $(3,4)$ be strictly positive. This example was found using a computer search and making use of the following fact: there exists a weight function $w:E\to \mathbb{R}_{\geq 0}$ simply inducing $\mathcal{P}$ with $w(e)=0$ if and only if the path system $\mathcal{P}/e$ is strictly metric.
Therefore, to find such an example, we consider graphs which remain non-s.m.\ after contracting an edge. For instance, if we take the path system $\mathcal{P}$ above and contract the persistent edge $(3,4)$ we obtain a consistent path system $\mathcal{P}/e$ on $G/e$. The path system $\mathcal{P}/e$ is precisely a non-s.m. path system defined on $K_{2,4}$, \cref{fig:graph1}. \\

{\bf Part 3:}
By induction on the number of edges. 
Wlog we may assume that every edge in $G$ is a path in $\mathcal{P}$. 
(A path system with this property is said to be {\em neighborly}, see \cite{CL}.)
If no edge in $\mathcal{P}$ is persistent, there is nothing to show,
so, let $e\in E$ be $\mathcal{P}$-persistent. 
As mentioned, contracting $e$ yields a consistent path system $\mathcal{P}/e$ on $G/e$.
But $G/e$ is strictly metrizable by
\cref{thm:minor_closed}. By induction, there is a weight function 
$w:E\setminus e \to \mathbb{R}_{\geq 0}$ simply inducing 
$\mathcal{P}/e$ s.t.\ $w(e')=0$ for every ($\mathcal{P}/e$)-persistent edge
$e'\in E\setminus e$. As mentioned above, an edge is persistent iff it appears
in all the trees $T_x$. Therefore, and edge in $E\setminus e$
is ($\mathcal{P}/e$)-persistent if only if it is $\mathcal{P}$-persistent. 
We extend $w$ to a weight function $\tilde{w}:E\to \mathbb{R}_{\geq 0}$, 
by setting $\tilde{w}(e)=0$ and $\tilde{w}(f) = w(f)$ for all $f\in E\setminus e$,
and argue that the weight function $\tilde{w}$ on $G$ simply induces $\mathcal{P}$.\\

Let $Q_{u,v}$ be a simple $(uv)$-path different from $P_{u,v}\in \mathcal{P}$,
for some $u,v\in V$. If $Q_{u,v}/e$ is a simple path then
\[\tilde{w}(P_{u,v})= w(P_{u,v}/e) <  w(Q_{u,v}/e) = \tilde{w}(Q_{u,v}).\]

On the other hand, if $Q_{u,v}/e$ is not a simple path, then $Q_{u,v}$ contains the 
vertices $x,y$, where $e=xy$, where $x$ and $y$ do not appear consecutively
in $Q_{u,v}$. We argue that $x$ and $y$ cannot even appear
two spots apart in $Q_{u,v}$, and 
are at distance at least $3$ along this path. If not,
there is a vertex $z\in Q_{u,v}$ adjacent to both $x$ and $y$. By assumption, both $zx$ 
and $zy$ are paths in $\mathcal{P}$, contrary to the assumption that 
$xy$ is persistent.

So $x$ and $y$ are of distance at least $3$ in $Q_{u,v}$, and it follows that $Q_{u,v}/e$ is the union of a simple path $\pi$ and a simple cycle $C$. Observe that every simple cycle in $G/e$ has positive $w$-weight. Indeed, let $f$ be an edge in $C$. Then
$C\setminus f$ is a simple path with the same endpoints as $f$, 
implying $w(f)<w(C\setminus f)$. In particular,\\ 
$w(C) = w(f) + w(C\setminus f)>0$. Since $Q_{u,v}/e = \pi+C$,
\[\tilde{w}(P_{u,v})= w(P_{u,v}/e) \leq   w(P)  < w(\pi)+w(C) = w(Q_{u,v}/e) = \tilde{w}(Q_{u,v}).\]

\end{proof}
\section{Discussion and Open Questions}\label{sec:open}
The famous graph minor theorem \cite{RS} says that every nontrivial
minor-closed family is characterized by a finite set of forbidden minors. 
So the most natural question to ask is which forbidden minors characterize strict metrizability. We raise the possibility that answer is to be found in \cref{fig:minor_zoo}. Namely,

\begin{conjecture}\label{conj:gang_of_six}
A graph is strictly metrizable if and only if it contains none of the six graphs in \cref{fig:minor_zoo} as a minor.
\end{conjecture}
Clearly, strictly metrizable graphs constitute a subclass of metrizable graphs.
But how are these two classes related to each other?
We showed in \cite{CCL} that, up to compliant edges, any $2$-connected
metrizable graph with at least $11$ vertices is a subdivision of $K_{2,3}$, $K_4$, $W_4$ and $W_4'$  or else $K_{2,n}$. A resolution to \cref{conj:gang_of_six} would
determine which subdivisions of these graphs are strictly metrizable.
Such a resolution would shed light on the structure of metrizable
graphs, and may possibly yield a practical algorithm to decide metrizability.
To elaborate on this last point: As shown in \cite{CL}, there exists a
polynomial-time algorithm to decide whether a given graph is metrizable.
However, we still do not have an honest-to-goodness
efficient algorithm for this decision problem.\\

In \cref{sec:persistent_edges}, we defined what it means for a non-negative weight function to simply induce a path system. But this notion makes sense even if negatively-weighted edges are allowed. 
Let us say a real valued weight function \mbox{$w:E\to \mathbb{R}$} simply induces a consistent path $\mathcal{P}$ on $G=(V,E)$ if every $P_{u,v}\in \mathcal{P}$ is the unique simple $(uv)$-geodesic w.r.t.\ to $w$. In fact, similar to \cref{lem:pers_vs_0}, it can be shown that if $w:E\to \mathbb{R}$ simply induces $\mathcal{P}$, then every non-positive edge is necessarily persistent. 
On the other hand, unlike path systems induced by non-negative weights, path systems induced by real weights need not be consistent. As an example, consider the $5$-cycle in which four edges weigh $2$ and the fifth one weighs $-3$. The shortest simple path between two adjacent vertices in the edge between them, while the shortest simple path connecting two non-adjacent vertices contains the edge of weight $-3$. It is not difficult to see that consistency fails here. While real edge weights don't necessarily induce consistent path systems, we ask whether they can help us extend the repertoire of consistent path systems.
Concretely, every simply induced consistent path system that
we have examined so far could also be induced by some non-negative weight
function. We wonder if there exist examples to the contrary:
\begin{open}\label{open:neg}
Let $\cal P$ be a consistent path system that
is simply induced by real (possibly non-positive) edge weights. Is it true that $\cal P$ can necessarily also be induced by positive edge weights?
\end{open}
\newpage
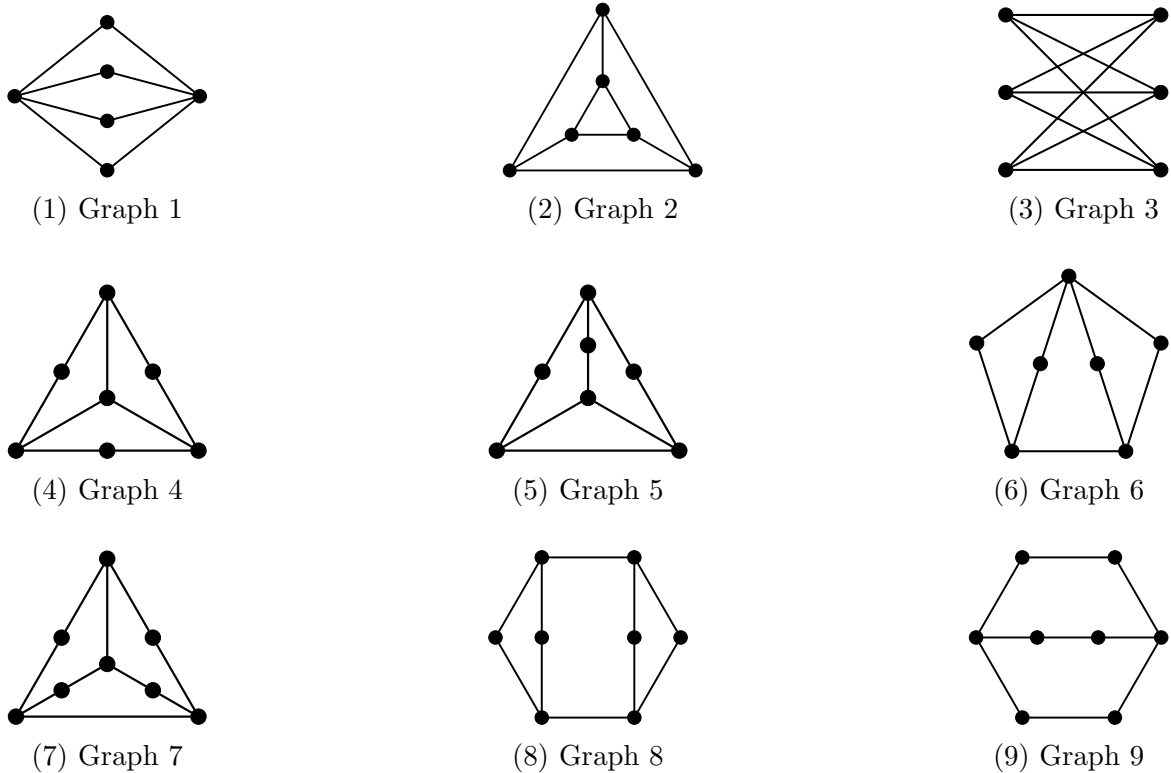
\begin{figure}[!htb]
        
        \renewcommand\thesubfigure{(\arabic{subfigure})}
	\centering
	\begin{subfigure}{0.175\textwidth}
		\centering
		\resizebox{\textwidth}{!}{
			\begin{tikzpicture}
			\def\r{5}
			\def\s{4}
			\node[draw,circle,fill] (1) at (-\r,0) {$1$};
			\node[draw,circle,fill] (2) at (\r,0) {$2$};
			\node[draw,circle,fill]  (3) at (0,\s) {$3$};
			\node[draw,circle,fill]  (4) at (0,\s/3) {$4$};
			\node[draw,circle,fill]  (5) at (0,-\s/3) {$5$};
			\node[draw,circle,fill]  (6) at (0,-\s) {$6$};

			\draw [line width=3pt,-] (1) -- (3);
			\draw [line width=3pt,-] (1) -- (4);
			\draw [line width=3pt,-] (1) -- (5);
			\draw [line width=3pt,-] (1) -- (6);
			\draw [line width=3pt,-] (2) -- (3);
			\draw [line width=3pt,-] (2) -- (4);
			\draw [line width=3pt,-] (2) -- (5);
			\draw [line width=3pt,-] (2) -- (6);

			\end{tikzpicture}
		}
		\caption{Graph 1}
		\label[graph]{fig:graph1}
	\end{subfigure}
	\hfill
	\begin{subfigure}{0.175\textwidth}
		\centering
		\resizebox{\textwidth}{!}{
			\begin{tikzpicture}

				\def\r{6}
				\def\s{2}
				\node[draw,circle,fill] (1) at (0*360/3 +90: \r) {$1$};
				\node[draw,circle,fill] (2) at (1*360/3 +90: \r) {$2$};
				\node[draw,circle,fill] (3) at (2*360/3 +90: \r) {$3$};
				\node[draw,circle,fill] (4) at (0*360/3 +90: \s) {$4$};
				\node[draw,circle,fill] (5) at (1*360/3 +90: \s) {$5$};
				\node[draw,circle,fill] (6) at (2*360/3 +90: \s) {$6$};

				\draw [line width=3pt,-] (1) -- (2);
				\draw [line width=3pt,-] (1) -- (3);
				\draw [line width=3pt,-] (2) -- (3);
				
				\draw [line width=3pt,-] (4) -- (5);
				\draw [line width=3pt,-] (5) -- (6);
				\draw [line width=3pt,-] (4) -- (6);
	
				\draw [line width=3pt,-] (1) -- (4);
				\draw [line width=3pt,-] (2) -- (5);
				\draw [line width=3pt,-] (3) -- (6);

			\end{tikzpicture}
		}
		\caption{Graph 2}
		\label[graph]{fig:graph2}
	\end{subfigure}
	\hfill
	\begin{subfigure}{0.15\textwidth}
		\centering
		\resizebox{\textwidth}{!}{
			\begin{tikzpicture}
				\def\r{4}
				\def\s{4}
				\node[draw,circle,fill] (1) at (-\r,\s) {$1$};
				\node[draw,circle,fill] (2) at (-\r,0) {$2$};
				\node[draw,circle,fill] (3) at (-\r,-\s) {$3$};
				\node[draw,circle,fill] (4) at (\r,\s) {$4$};
				\node[draw,circle,fill] (5) at (\r,0) {$5$};
				\node[draw,circle,fill] (6) at (\r,-\s) {$6$};

				\draw [line width=3pt,-] (1) -- (4);
				\draw [line width=3pt,-] (1) -- (5);
				\draw [line width=3pt,-] (1) -- (6);
				\draw [line width=3pt,-] (2) -- (4);
				\draw [line width=3pt,-] (2) -- (5);
				\draw [line width=3pt,-] (2) -- (6);
				\draw [line width=3pt,-] (3) -- (4);
				\draw [line width=3pt,-] (3) -- (5);
				\draw [line width=3pt,-] (3) -- (6);

			\end{tikzpicture}
		}
		\caption{Graph 3}
		\label[graph]{fig:graph3}
	\end{subfigure}

	\vspace{5mm}

	\begin{subfigure}{0.175\textwidth}
		\centering
		\resizebox{\textwidth}{!}{
			\begin{tikzpicture}
			
			\def\r{5}
			\node[draw,circle,fill] (1) at (0*360/3 +90: \r) {$1$};
			\node[draw,circle,fill] (3) at (1*360/3 +90: \r) {$3$};
			\node[draw,circle,fill] (5) at (2*360/3 +90: \r) {$5$};
			\node[draw,circle,fill] (2) at ($(1)!.5!(3)$) {$2$};
			\node[draw,circle,fill] (4) at ($(3)!.5!(5)$) {$4$};
			\node[draw,circle,fill] (6) at ($(5)!.5!(1)$) {$6$};
			\node[draw,circle,fill] (7) at (0,0) {$7$};

			\draw [line width=3pt,-] (1) -- (2);
			\draw [line width=3pt,-] (1) -- (6);
			\draw [line width=3pt,-] (2) -- (3);
			\draw [line width=3pt,-] (3) -- (4);
			\draw [line width=3pt,-] (4) -- (5);
			\draw [line width=3pt,-] (5) -- (6);
			\draw [line width=3pt,-] (1) -- (7);
			\draw [line width=3pt,-] (3) -- (7);
			\draw [line width=3pt,-] (5) -- (7);

			\end{tikzpicture}
		}
		\caption{Graph 4}
		\label[graph]{fig:graph4}
	\end{subfigure}
	\hfill
	\begin{subfigure}{0.175\textwidth}
		\centering
		\resizebox{\textwidth}{!}{
			\begin{tikzpicture}
			
				\def\r{5}
				\node[draw,circle,fill] (1) at (0*360/3 +90: \r) {$1$};
				\node[draw,circle,fill] (3) at (1*360/3 +90: \r) {$3$};
				\node[draw,circle,fill] (4) at (2*360/3 +90: \r) {$4$};
				\node[draw,circle,fill] (2) at ($(1)!.5!(3)$) {$2$};
				\node[draw,circle,fill] (5) at ($(1)!.5!(4)$) {$5$};
				\node[draw,circle,fill] (7) at (0,0) {$7$};
				\node[draw,circle,fill] (6) at ($(1)!.5!(7)$) {$6$};

				\draw [line width=3pt,-] (1) -- (2);
				\draw [line width=3pt,-] (1) -- (5);
				\draw [line width=3pt,-] (1) -- (6);
				\draw [line width=3pt,-] (2) -- (3);
				\draw [line width=3pt,-] (3) -- (4);
				\draw [line width=3pt,-] (3) -- (7);
				\draw [line width=3pt,-] (4) -- (5);
				\draw [line width=3pt,-] (4) -- (7);
				\draw [line width=3pt,-] (6) -- (7);

			\end{tikzpicture}
		}
		\caption{Graph 5}
		\label[graph]{fig:graph5}
	\end{subfigure}
	\hfill
	\begin{subfigure}{0.175\textwidth}
		\centering
		\resizebox{\textwidth}{!}{
			\begin{tikzpicture}

			\node[draw,circle,fill] (1) at (0*360/5 +90: 5cm) {$1$};
			\node[draw,circle,fill] (2) at (1*360/5 +90: 5cm) {$2$};
			\node[draw,circle,fill] (3) at (2*360/5 +90: 5cm) {$3$};
			\node[draw,circle,fill] (4) at (3*360/5 +90: 5cm) {$4$};
			\node[draw,circle,fill] (5) at (4*360/5 +90: 5cm) {$5$};
			
			\node[draw,circle,fill] (6) at ($(1)!0.5!(3)$) {$6$};
			\node[draw,circle,fill] (7) at  ($(1)!0.5!(4)$) {$7$};

			\draw [line width=3pt,-] (1) -- (2);
			\draw [line width=3pt,-] (2) -- (3);
			\draw [line width=3pt,-] (3) -- (4);
			\draw [line width=3pt,-] (4) -- (5);
			\draw [line width=3pt,-] (5) -- (1);
			\draw [line width=3pt,-] (1) -- (6);
			\draw [line width=3pt,-] (3) -- (6);
			\draw [line width=3pt,-] (1) -- (7);
			\draw [line width=3pt,-] (4) -- (7);
			
			\end{tikzpicture}
		}
		\caption{Graph 6}
		\label[graph]{fig:graph6}
	\end{subfigure}

	\vspace{5mm}

	\begin{subfigure}{0.175\textwidth}
		\centering
		\resizebox{\textwidth}{!}{
			\begin{tikzpicture}
			
				\def\r{5}
				\node[draw,circle,fill] (1) at (0*360/3 +90: \r) {$1$};
				\node[draw,circle,fill] (3) at (1*360/3 +90: \r) {$3$};
				\node[draw,circle,fill] (4) at (2*360/3 +90: \r) {$4$};
				\node[draw,circle,fill] (2) at ($(1)!.5!(3)$) {$2$};
				\node[draw,circle,fill] (5) at ($(1)!.5!(4)$) {$5$};
				\node[draw,circle,fill] (8) at (0,0) {$8$};
				\node[draw,circle,fill] (6) at ($(3)!.5!(8)$) {$6$};
				\node[draw,circle,fill] (7) at ($(4)!.5!(8)$) {$7$};

				\draw [line width=3pt,-] (1) -- (2);
				\draw [line width=3pt,-] (1) -- (5);
				\draw [line width=3pt,-] (1) -- (8);
				\draw [line width=3pt,-] (2) -- (3);
				\draw [line width=3pt,-] (3) -- (4);
				\draw [line width=3pt,-] (3) -- (6);
				\draw [line width=3pt,-] (4) -- (5);
				\draw [line width=3pt,-] (4) -- (7);
				\draw [line width=3pt,-] (6) -- (8);
				\draw [line width=3pt,-] (7) -- (8);

			\end{tikzpicture}
		}
		\caption{Graph 7}
		\label[graph]{fig:graph7}
	\end{subfigure}
	\hfill
	\begin{subfigure}{0.175\textwidth}
		\centering
		\resizebox{\textwidth}{!}{
			\begin{tikzpicture}
				\node[draw,circle,fill] (1) at (0*360/6 +120: 5cm) {$1$};
				\node[draw,circle,fill] (2) at (1*360/6 +120: 5cm) {$2$};
				\node[draw,circle,fill] (3) at (2*360/6 +120: 5cm) {$3$};
	
				\node[draw,circle,fill] (5) at (3*360/6 +120: 5cm) {$5$};
				\node[draw,circle,fill] (6) at (4*360/6 +120: 5cm) {$6$};
				\node[draw,circle,fill] (7) at (5*360/6 +120: 5cm) {$7$};
				
				\node[draw,circle,fill] (4) at ($(1)!.5!(3)$) {$4$};
				\node[draw,circle,fill] (8) at ($(5)!.5!(7)$) {$8$};
				
				\draw [line width=3pt,-] (1) -- (2);
				\draw [line width=3pt,-] (1) -- (7);
				\draw [line width=3pt,-] (1) -- (4);
				\draw [line width=3pt,-] (2) -- (3);
				\draw [line width=3pt,-] (3) -- (5);
				\draw [line width=3pt,-] (3) -- (4);
				\draw [line width=3pt,-] (5) -- (6);
				\draw [line width=3pt,-] (5) -- (8);
				\draw [line width=3pt,-] (6) -- (7);
				\draw [line width=3pt,-] (7) -- (8);

			\end{tikzpicture}
		}
		\caption{Graph 8}
		\label[graph]{fig:graph8}
	\end{subfigure}
	\hfill
	\begin{subfigure}{0.175\textwidth}
		\centering
		\resizebox{\textwidth}{!}{
			\begin{tikzpicture}
				\node[draw,circle,fill] (1) at (0*360/6 +120: 5cm) {$1$};
				\node[draw,circle,fill] (2) at (1*360/6 +120: 5cm) {$2$};
				\node[draw,circle,fill] (3) at (2*360/6 +120: 5cm) {$3$};
				\node[draw,circle,fill] (4) at (3*360/6 +120: 5cm) {$4$};
				\node[draw,circle,fill] (5) at (4*360/6 +120: 5cm) {$5$};
				\node[draw,circle,fill] (6) at (5*360/6 +120: 5cm) {$6$};
				\node[draw,circle,fill] (7) at ($(2)!.33!(5)$) {$7$};
				\node[draw,circle,fill] (8) at ($(2)!.66!(5)$) {$8$};
				
				\draw [line width=3pt,-] (1) -- (2);
				\draw [line width=3pt,-] (1) -- (6);
				\draw [line width=3pt,-] (2) -- (3);
				\draw [line width=3pt,-] (2) -- (7);
				\draw [line width=3pt,-] (3) -- (4);
				\draw [line width=3pt,-] (4) -- (5);
				\draw [line width=3pt,-] (5) -- (6);
				\draw [line width=3pt,-] (5) -- (8);
				\draw [line width=3pt,-] (7) -- (8);

			\end{tikzpicture}
		}
		\caption{Graph 9}
		\label[graph]{fig:graph9}
	\end{subfigure}
	\caption{Currently known topologically minimal non-s.m.\ graphs}
	\label{fig:zoo}
\end{figure}
\vspace{10mm}
\begin{figure}[!htb]
	\centering
	\begin{subfigure}{0.175\textwidth}
		\centering
		\resizebox{\textwidth}{!}{
			\begin{tikzpicture}
			\def\r{5}
			\def\s{4}
			\node[draw,circle,fill] (1) at (-\r,0) {$1$};
			\node[draw,circle,fill] (2) at (\r,0) {$2$};
			\node[draw,circle,fill]  (3) at (0,\s) {$3$};
			\node[draw,circle,fill]  (4) at (0,\s/3) {$4$};
			\node[draw,circle,fill]  (5) at (0,-\s/3) {$5$};
			\node[draw,circle,fill]  (6) at (0,-\s) {$6$};

			\draw [line width=3pt,-] (1) -- (3);
			\draw [line width=3pt,-] (1) -- (4);
			\draw [line width=3pt,-] (1) -- (5);
			\draw [line width=3pt,-] (1) -- (6);
			\draw [line width=3pt,-] (2) -- (3);
			\draw [line width=3pt,-] (2) -- (4);
			\draw [line width=3pt,-] (2) -- (5);
			\draw [line width=3pt,-] (2) -- (6);

			\end{tikzpicture}
		}
		\caption{Graph 1}
		\label[graph]{fig:minor_graph1}
	\end{subfigure}
	\hfill
	\begin{subfigure}{0.175\textwidth}
		\centering
		\resizebox{\textwidth}{!}{
			\begin{tikzpicture}

				\def\r{6}
				\def\s{2}
				\node[draw,circle,fill] (1) at (0*360/3 +90: \r) {$1$};
				\node[draw,circle,fill] (2) at (1*360/3 +90: \r) {$2$};
				\node[draw,circle,fill] (3) at (2*360/3 +90: \r) {$3$};
				\node[draw,circle,fill] (4) at (0*360/3 +90: \s) {$4$};
				\node[draw,circle,fill] (5) at (1*360/3 +90: \s) {$5$};
				\node[draw,circle,fill] (6) at (2*360/3 +90: \s) {$6$};

				\draw [line width=3pt,-] (1) -- (2);
				\draw [line width=3pt,-] (1) -- (3);
				\draw [line width=3pt,-] (2) -- (3);
				
				\draw [line width=3pt,-] (4) -- (5);
				\draw [line width=3pt,-] (5) -- (6);
				\draw [line width=3pt,-] (4) -- (6);
	
				\draw [line width=3pt,-] (1) -- (4);
				\draw [line width=3pt,-] (2) -- (5);
				\draw [line width=3pt,-] (3) -- (6);

			\end{tikzpicture}
		}
		\caption{Graph 2}
		\label[graph]{fig:minor_graph2}
	\end{subfigure}
	\hfill
	\begin{subfigure}{0.15\textwidth}
		\centering
		\resizebox{\textwidth}{!}{
			\begin{tikzpicture}
				\def\r{4}
				\def\s{4}
				\node[draw,circle,fill] (1) at (-\r,\s) {$1$};
				\node[draw,circle,fill] (2) at (-\r,0) {$2$};
				\node[draw,circle,fill] (3) at (-\r,-\s) {$3$};
				\node[draw,circle,fill] (4) at (\r,\s) {$4$};
				\node[draw,circle,fill] (5) at (\r,0) {$5$};
				\node[draw,circle,fill] (6) at (\r,-\s) {$6$};

				\draw [line width=3pt,-] (1) -- (4);
				\draw [line width=3pt,-] (1) -- (5);
				\draw [line width=3pt,-] (1) -- (6);
				\draw [line width=3pt,-] (2) -- (4);
				\draw [line width=3pt,-] (2) -- (5);
				\draw [line width=3pt,-] (2) -- (6);
				\draw [line width=3pt,-] (3) -- (4);
				\draw [line width=3pt,-] (3) -- (5);
				\draw [line width=3pt,-] (3) -- (6);

			\end{tikzpicture}
		}
		\caption{Graph 3}
		\label[graph]{fig:minor_graph3}
	\end{subfigure}

	\vspace{5mm}

	\begin{subfigure}{0.175\textwidth}
		\centering
		\resizebox{\textwidth}{!}{
			\begin{tikzpicture}
			
			\def\r{5}
			\node[draw,circle,fill] (1) at (0*360/3 +90: \r) {$1$};
			\node[draw,circle,fill] (3) at (1*360/3 +90: \r) {$3$};
			\node[draw,circle,fill] (5) at (2*360/3 +90: \r) {$5$};
			\node[draw,circle,fill] (2) at ($(1)!.5!(3)$) {$2$};
			\node[draw,circle,fill] (4) at ($(3)!.5!(5)$) {$4$};
			\node[draw,circle,fill] (6) at ($(5)!.5!(1)$) {$6$};
			\node[draw,circle,fill] (7) at (0,0) {$7$};

			\draw [line width=3pt,-] (1) -- (2);
			\draw [line width=3pt,-] (1) -- (6);
			\draw [line width=3pt,-] (2) -- (3);
			\draw [line width=3pt,-] (3) -- (4);
			\draw [line width=3pt,-] (4) -- (5);
			\draw [line width=3pt,-] (5) -- (6);
			\draw [line width=3pt,-] (1) -- (7);
			\draw [line width=3pt,-] (3) -- (7);
			\draw [line width=3pt,-] (5) -- (7);

			\end{tikzpicture}
		}
		\caption{Graph 4}
		\label[graph]{fig:minor_graph4}
	\end{subfigure}
	\hfill
	\begin{subfigure}{0.175\textwidth}
		\centering
		\resizebox{\textwidth}{!}{
			\begin{tikzpicture}
			
				\def\r{5}
				\node[draw,circle,fill] (1) at (0*360/3 +90: \r) {$1$};
				\node[draw,circle,fill] (3) at (1*360/3 +90: \r) {$3$};
				\node[draw,circle,fill] (4) at (2*360/3 +90: \r) {$4$};
				\node[draw,circle,fill] (2) at ($(1)!.5!(3)$) {$2$};
				\node[draw,circle,fill] (5) at ($(1)!.5!(4)$) {$5$};
				\node[draw,circle,fill] (7) at (0,0) {$7$};
				\node[draw,circle,fill] (6) at ($(1)!.5!(7)$) {$6$};

				\draw [line width=3pt,-] (1) -- (2);
				\draw [line width=3pt,-] (1) -- (5);
				\draw [line width=3pt,-] (1) -- (6);
				\draw [line width=3pt,-] (2) -- (3);
				\draw [line width=3pt,-] (3) -- (4);
				\draw [line width=3pt,-] (3) -- (7);
				\draw [line width=3pt,-] (4) -- (5);
				\draw [line width=3pt,-] (4) -- (7);
				\draw [line width=3pt,-] (6) -- (7);

			\end{tikzpicture}
		}
		\caption{Graph 5}
		\label[graph]{fig:minor_graph5}
	\end{subfigure}
	\hfill
	\begin{subfigure}{0.175\textwidth}
		\centering
		\resizebox{\textwidth}{!}{
			\begin{tikzpicture}
			
				\node[draw,circle,fill] (1) at (0*360/6 +120: 5cm) {$1$};
				\node[draw,circle,fill] (2) at (1*360/6 +120: 5cm) {$2$};
				\node[draw,circle,fill] (3) at (2*360/6 +120: 5cm) {$3$};
				\node[draw,circle,fill] (4) at (3*360/6 +120: 5cm) {$4$};
				\node[draw,circle,fill] (5) at (4*360/6 +120: 5cm) {$5$};
				\node[draw,circle,fill] (6) at (5*360/6 +120: 5cm) {$6$};
				\node[draw,circle,fill] (7) at ($(2)!.33!(5)$) {$7$};
				\node[draw,circle,fill] (8) at ($(2)!.66!(5)$) {$8$};
				
				\draw [line width=3pt,-] (1) -- (2);
				\draw [line width=3pt,-] (1) -- (6);
				\draw [line width=3pt,-] (2) -- (3);
				\draw [line width=3pt,-] (2) -- (7);
				\draw [line width=3pt,-] (3) -- (4);
				\draw [line width=3pt,-] (4) -- (5);
				\draw [line width=3pt,-] (5) -- (6);
				\draw [line width=3pt,-] (5) -- (8);
				\draw [line width=3pt,-] (7) -- (8);

			\end{tikzpicture}
		}
		\caption{Graph 6}
		\label[graph]{fig:minor_graph6}
	\end{subfigure}
	
	\caption{All currently known minor minimal non-s.m. graphs}
	\label{fig:minor_zoo}
\end{figure}
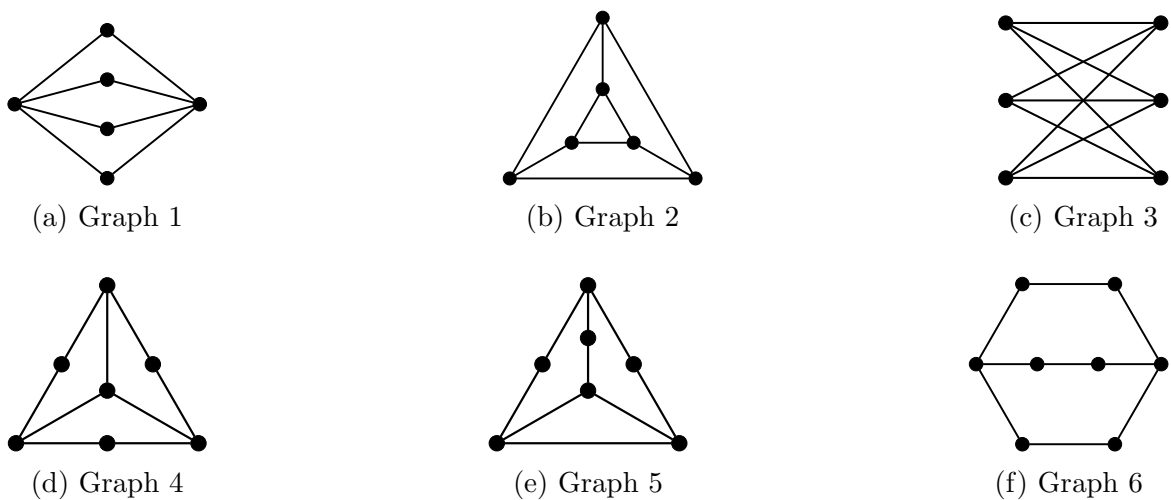
\appendix \label{append:certificates}
\section{Certificates of non-strict metrizability}\label{append:certificates}
	For each graph $G$ in \Cref{fig:zoo} we give a path system in $G$ along with a system of inequalities a weight function strictly inducing this path system must satisfy. In each case, adding these inequalities together implies $0<0$, a contradiction.\\
	\vspace{5mm}
	
	\noindent
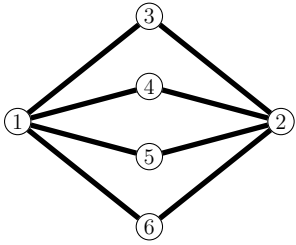
	\begin{minipage}[c]{0.4\textwidth}
		\centering
		\begin{tikzpicture}[scale=0.35, every node/.style={scale=0.35}]
		
		\def \r{5}
		\def \s{4}
		
		\node[draw,circle,minimum size=.5cm,inner sep=1pt] (1) at (-\r,0) [scale = 2]{$1$};
		\node[draw,circle,minimum size=.5cm,inner sep=1pt] (2) at (\r,0) [scale = 2]{$2$};
		\node[draw,circle,minimum size=.5cm,inner sep=1pt] (3) at (0,\s)[scale = 2] {$3$};
		\node[draw,circle,minimum size=.5cm,inner sep=1pt] (4) at (0,\s/3)[scale = 2] {$4$};
		\node[draw,circle,minimum size=.5cm,inner sep=1pt] (5) at (0,-\s/3)[scale = 2] {$5$};
		\node[draw,circle,minimum size=.5cm,inner sep=1pt] (6) at (0,-\s)[scale = 2] {$6$};

		\draw [line width=2pt,-] (1) -- (3);
		\draw [line width=2pt,-] (1) -- (4);
		\draw [line width=2pt,-] (1) -- (5);
		\draw [line width=2pt,-] (1) -- (6);
		\draw [line width=2pt,-] (2) -- (3);
		\draw [line width=2pt,-] (2) -- (4);
		\draw [line width=2pt,-] (2) -- (5);
		\draw [line width=2pt,-] (2) -- (6);

		\end{tikzpicture}
	\end{minipage}
	\begin{minipage}[c]{0.4\textwidth}
		\centering
		\small
		\begin{equation*}
		\begin{gathered}
		(1,3,2),\ (3,1,4),\ (3,2,5),\ (3,1,6),
	 \\ (4,1,5),\  (4,2,6),  \  (5,1,6)
		\end{gathered}
		\end{equation*}
	\end{minipage}
	\\
	\vspace{4mm}
	\begin{minipage}{\textwidth}
		\small
		\begin{equation*}
		\begin{split}
		w_{2,3} + w_{2,5} & < w_{1,3} + w_{1,5}\\
		w_{1,4} + w_{1,5} & <w_{2,4} + w_{2,5}\\
		w_{2,4} + w_{2,6} & < w_{1,4} + w_{1,6}\\
		w_{1,3} + w_{1,6}  & < w_{2,3} + w_{2,6}
		\end{split}
		\end{equation*}
		
	\end{minipage}
	
	\begin{center}
		\line(1,0){400}
	\end{center}
		\begin{minipage}[c]{0.4\textwidth}
			\centering
			\begin{tikzpicture}[scale=0.6, every node/.style={scale=0.6}]
			\def\r{3}
			\def\s{1}
			\node[draw,circle,minimum size=.5cm,inner sep=1pt] (1) at (0*360/3 +90: \r) {$1$};
			\node[draw,circle,minimum size=.5cm,inner sep=1pt] (2) at (1*360/3 +90: \r) {$2$};
			\node[draw,circle,minimum size=.5cm,inner sep=1pt] (3) at (2*360/3 +90: \r) {$3$};
			\node[draw,circle,minimum size=.5cm,inner sep=1pt] (4) at (0*360/3 +90: \s) {$4$};
			\node[draw,circle,minimum size=.5cm,inner sep=1pt] (5) at (1*360/3 +90: \s) {$5$};
			\node[draw,circle,minimum size=.5cm,inner sep=1pt] (6) at (2*360/3 +90: \s) {$6$};

			\draw [line width=2pt,-] (1) -- (2);
			\draw [line width=2pt,-] (1) -- (3);
			\draw [line width=2pt,-] (2) -- (3);
			
			\draw [line width=2pt,-] (4) -- (5);
			\draw [line width=2pt,-] (5) -- (6);
			\draw [line width=2pt,-] (4) -- (6);
	
			\draw [line width=2pt,-] (1) -- (4);
			\draw [line width=2pt,-] (2) -- (5);
			\draw [line width=2pt,-] (3) -- (6);
			
			\end{tikzpicture}
		\end{minipage}
		\begin{minipage}[c]{0.4\textwidth}
			\centering
			\small
			\begin{equation*}
			\begin{gathered}
			 (1,2,5), \ (1,4,6), \ (2,5,4),  \\
			 (2,3,6), \ (3,1,4), \ (3,6,5)
			\end{gathered}
			\end{equation*}
		\end{minipage}
		\\
		\vspace{5mm}
		\begin{minipage}{\textwidth}
			\small
			\begin{equation*}
			\begin{split}
			w_{1,2} + w_{2,5}  & < w_{1,4} +  w_{4,5}\\
			w_{1,4} + w_{4,6}  & < w_{1,3} +  w_{3,6}\\
			w_{2,3} + w_{3,6}  & < w_{2,5} +  w_{5,6}\\
			w_{2,5} + w_{4,5}  & < w_{1,2} +  w_{1,4}\\
			w_{1,3} + w_{1,4}  & < w_{3,6} +  w_{4,6}\\
			w_{3,6} + w_{5,6}  & < w_{2,3} +  w_{2,5}
			\end{split}
			\end{equation*}
		\end{minipage}
		\begin{center}
			\line(1,0){400}
		\end{center}
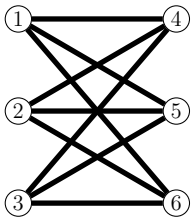
	\begin{minipage}[c]{0.4\textwidth}
		\centering
		\begin{tikzpicture}[scale=0.35, every node/.style={scale=0.35}]
		\def\r{3.5}
		\def\s{3}
		
		\node[draw,circle,minimum size=.5cm,inner sep=1pt] (1) at (-\s,\r) [scale=2]{$1$};
		\node[draw,circle,minimum size=.5cm,inner sep=1pt] (2) at (-\s,0) [scale=2]{$2$};
		\node[draw,circle,minimum size=.5cm,inner sep=1pt] (3) at (-\s,-\r) [scale=2] {$3$};

		\node[draw,circle,minimum size=.5cm,inner sep=1pt] (4) at (\s,\r) [scale=2]{$4$};
		\node[draw,circle,minimum size=.5cm,inner sep=1pt] (5) at (\s,0) [scale=2]{$5$};
		\node[draw,circle,minimum size=.5cm,inner sep=1pt] (6) at (\s,-\r) [scale=2] {$6$};

		\draw [line width=2pt,-] (1) -- (4);
		\draw [line width=2pt,-] (1) -- (5);
		\draw [line width=2pt,-] (1) -- (6);
		\draw [line width=2pt,-] (2) -- (4);
		\draw [line width=2pt,-] (2) -- (5);
		\draw [line width=2pt,-] (2) -- (6);
		\draw [line width=2pt,-] (3) -- (4);
		\draw [line width=2pt,-] (3) -- (5);
		\draw [line width=2pt,-] (3) -- (6);

		\end{tikzpicture}
	\end{minipage}
	\begin{minipage}[c]{0.4\textwidth}
		\centering
		\small
		\begin{equation*}
		\begin{gathered}
		(1,4,2), \ (1,6,3),\ (2,6,3), \\ 
        (4,3,5),\ (4,3,6),\ (5,1,6), 
		\end{gathered}
		\end{equation*}
	\end{minipage}
	\\
	\vspace{5mm}
	\begin{minipage}{\textwidth}
		\small
		\begin{equation*}
		\begin{split}
		w_{2,6} + w_{3,6} & < w_{2,4} + w_{3,4}\\
		w_{1,5} + w_{1,6} & < w_{3,5} + w_{3,6}\\
		w_{1,4} + w_{2,4} & < w_{1,6} + w_{2,6}\\
		w_{3,4} + w_{3,5}  & < w_{1,4} + w_{1,5}
		\end{split}
		\end{equation*}
	\end{minipage}
	\begin{center}
		\line(1,0){400}
	\end{center}

	  \begin{minipage}[c]{0.4\textwidth}
		\centering
		\begin{tikzpicture}[scale=0.34, every node/.style={scale=0.34}]

		\node[draw,circle,minimum size=.5cm,inner sep=1pt] (1) at (0*360/3 +90: 5cm)[scale =2] {$1$};
		\node[draw,circle,minimum size=.5cm,inner sep=1pt] (3) at (1*360/3 +90: 5cm)[scale =2] {$3$};
		\node[draw,circle,minimum size=.5cm,inner sep=1pt] (5) at (2*360/3 +90: 5cm) [scale =2]{$5$};
		\node[draw,circle,minimum size=.5cm,inner sep=1pt] (2) at ($(1)!0.5!(3)$)[scale =2] {$2$};
		\node[draw,circle,minimum size=.5cm,inner sep=1pt] (4) at ($(3)!0.5!(5)$)[scale =2] {$4$};
		\node[draw,circle,minimum size=.5cm,inner sep=1pt] (6) at ($(1)!0.5!(5)$)[scale =2] {$6$};
		\node[draw,circle,minimum size=.5cm,inner sep=1pt] (7) at (0,0)[scale =2] {$7$};

		\draw [line width=2pt,-] (1) -- (2);
		\draw [line width=2pt,-] (1) -- (6);
		\draw [line width=2pt,-] (1) -- (7);
		\draw [line width=2pt,-] (2) -- (3);
		\draw [line width=2pt,-] (3) -- (4);
		\draw [line width=2pt,-] (3) -- (7);
		\draw [line width=2pt,-] (4) -- (5);
		\draw [line width=2pt,-] (5) -- (6);
		\draw [line width=2pt,-] (5) -- (7);

		\end{tikzpicture}
	\end{minipage}
	\begin{minipage}[c]{0.4\textwidth}
		\centering
		\small
		\begin{equation*}
		\begin{gathered}
		(1,2,3), \ (1,6,5,4), \ (1,6,5),\ (2,3,4),  \\
        (2,3,4,5),\ (2,1,6),\ (2,1,7),\  (3,4,5), \\ 
        (3,2,1,6),\ (4,5,6),\ (4,3,7),\ (6,5,7)
		\end{gathered}
		\end{equation*}
	\end{minipage}
	\\
	\vspace{5mm}
	\begin{minipage}{\textwidth}
		\small
		\begin{equation*}
		\begin{split}
	 w_{3,4}+w_{3,7} & < w_{4,5} + w_{5,7}\\
		w_{1,2} + w_{1,7} & <  w_{2,3} + w_{3,7}\\
		w_{5,6} + w_{5,7} & < w_{1,6} + w_{1,7}\\
		w_{1,6}+w_{5,6} + w_{4,5}  & < w_{1,2} + w_{2,3} + w_{3,4}\\
		w_{2,3}+w_{3,4}  + w_{4,5} & < w_{1,2} + w_{1,6} + w_{5,6}\\
		w_{2,3} + w_{1,2} + w_{1,6} & < w_{3,4} + w_{4,5} + w_{5,6}
		\end{split}
		\end{equation*}
	\end{minipage}

	\begin{center}
		\line(1,0){400}
	\end{center}
	 \begin{minipage}[c]{0.4\textwidth}
		\centering
		\begin{tikzpicture}[scale=0.34, every node/.style={scale=0.34}]

		\node[draw,circle,minimum size=.5cm,inner sep=1pt] (1) at (0*360/3 +90: 5cm) [scale =2] {$1$};
		\node[draw,circle,minimum size=.5cm,inner sep=1pt] (3) at (1*360/3 +90: 5cm) [scale =2] {$3$};
		\node[draw,circle,minimum size=.5cm,inner sep=1pt] (4) at (2*360/3 +90: 5cm) [scale =2] {$4$};
		\node[draw,circle,minimum size=.5cm,inner sep=1pt] (7) at (0,0) [scale =2] {$7$};
		\node[draw,circle,minimum size=.5cm,inner sep=1pt] (2) at ($(1)!.5!(3)$)  [scale =2]{$2$};
		\node[draw,circle,minimum size=.5cm,inner sep=1pt] (5) at ($(1)!.5!(4)$) [scale =2] {$5$};
		\node[draw,circle,minimum size=.5cm,inner sep=1pt] (6) at ($(1)!.5!(7)$) [scale =2] {$6$};

		\draw [line width=2pt,-] (1) -- (2);
		\draw [line width=2pt,-] (1) -- (5);
		\draw [line width=2pt,-] (1) -- (6);
		\draw [line width=2pt,-] (2) -- (3);
		\draw [line width=2pt,-] (3) -- (4);
		\draw [line width=2pt,-] (3) -- (7);
		\draw [line width=2pt,-] (4) -- (5);
		\draw [line width=2pt,-] (4) -- (7);
		\draw [line width=2pt,-] (6) -- (7);

		\end{tikzpicture}
	\end{minipage}
	\begin{minipage}[c]{0.4\textwidth}
		\centering
		\small
		\begin{equation*}
		\begin{gathered}
			(1, 2, 3), \ (1, 5, 4), \  (1, 6, 7), \ (2, 1, 5,4), \\
			(2, 1, 5), \ (2, 1, 6), \ (2, 3, 7),\ (3, 4, 5),  \\ 
              (3, 2, 1, 6),\ (4, 7, 6),  \ (5, 1, 6), \ (5, 1, 6, 7)
		\end{gathered}
		\end{equation*}
	\end{minipage}
	\\
	\vspace{4mm}
	\begin{minipage}{\textwidth}
		\small
		\begin{equation*}
		\begin{split}
		w_{2,3} + w_{1,2} + w_{1,6}  & < w_{3,7} + w_{6,7}\\
		w_{1,5} + w_{1,6} + w_{6,7}  & < w_{4,5} + w_{4,7}\\
		w_{1,2} + w_{1,5} + w_{4,5} & < w_{2,3} + w_{3,4}\\
		w_{4,7} + w_{6,7} & < w_{4,5} + w_{1,5} + w_{1,6}\\
		w_{3,4} + w_{4,5} &< w_{2,3} + w_{1,2} + w_{1,5}\\
		  w_{2,3} + w_{3,7} &< w_{1,2} + w_{1,6} + w_{6,7}
		\end{split}
		\end{equation*}
	\end{minipage}
	\begin{center}
		\line(1,0){400}
	\end{center}
	\begin{minipage}[c]{0.4\textwidth}
		\centering
		\begin{tikzpicture}[scale=0.35, every node/.style={scale=0.35}]

		\node[draw,circle,minimum size=.5cm,inner sep=1pt] (1) at (0*360/5 +90: 5cm) [scale=2]{$1$};
		\node[draw,circle,minimum size=.5cm,inner sep=1pt] (2) at (1*360/5 +90: 5cm) [scale=2]{$2$};
		\node[draw,circle,minimum size=.5cm,inner sep=1pt] (3) at (2*360/5 +90: 5cm) [scale=2]{$3$};
		\node[draw,circle,minimum size=.5cm,inner sep=1pt] (4) at (3*360/5 +90: 5cm) [scale=2]{$4$};
		\node[draw,circle,minimum size=.5cm,inner sep=1pt] (5) at (4*360/5 +90: 5cm) [scale=2]{$5$};
		
		\node[draw,circle,minimum size=.5cm,inner sep=1pt] (6) at ($(1)!0.5!(3)$)[scale=2] {$6$};
		\node[draw,circle,minimum size=.5cm,inner sep=1pt] (7) at  ($(1)!0.5!(4)$) [scale=2]{$7$};

		\draw [line width=2pt,-] (1) -- (2);
		\draw [line width=2pt,-] (2) -- (3);
		\draw [line width=2pt,-] (3) -- (4);
		\draw [line width=2pt,-] (4) -- (5);
		\draw [line width=2pt,-] (5) -- (1);
		\draw [line width=2pt,-] (1) -- (6);
		\draw [line width=2pt,-] (3) -- (6);
		\draw [line width=2pt,-] (1) -- (7);
		\draw [line width=2pt,-] (4) -- (7);
		
		\end{tikzpicture}
	\end{minipage}
	\begin{minipage}[c]{0.4\textwidth}
		\small
		\begin{equation*}
		\begin{gathered}
			(1, 6, 3), \ (1, 7, 4), \ (2, 3, 4), \ (2, 1, 5), \\
			 (2, 3, 6), \ (2, 3, 4, 7), \  (3, 4, 5), \ (3, 4, 7), \\
			  (4, 3, 6), \ (5, 4, 3, 6), \ (5, 4, 7),\ (6, 1, 7)
		\end{gathered}
		\end{equation*}
	\end{minipage}
	\\
	\vspace{4mm}
	\begin{minipage}{\textwidth}
		\small
		\begin{equation*}
		\begin{split}
			w_{1,2} + w_{1,5}& < w_{2,3} + w_{3,4} + w_{4,5}\\
			w_{2,3} + w_{3,4} + w_{4,7}& < w_{1,2} + w_{1,7}\\
			w_{4,5}+w_{3,4} + w_{3,6}& < w_{1,5} + w_{1,6}\\
			w_{1,6} + w_{1,7}& < w_{3,6} + w_{3,4} + w_{4,7}
		\end{split}
		\end{equation*}
	\end{minipage}
	\begin{center}
		\line(1,0){400}
	\end{center}
    
		\begin{minipage}[c]{0.4\textwidth}
			\centering
			\begin{tikzpicture}[scale=0.35, every node/.style={scale=0.35}]

			\node[draw,circle,minimum size=.5cm,inner sep=1pt] (1) at (0*360/3 +90: 5cm) [scale =2]{$1$};
			\node[draw,circle,minimum size=.5cm,inner sep=1pt] (3) at (1*360/3 +90: 5cm) [scale =2]{$3$};
			\node[draw,circle,minimum size=.5cm,inner sep=1pt] (4) at (2*360/3 +90: 5cm) [scale =2]{$4$};
			\node[draw,circle,minimum size=.5cm,inner sep=1pt] (2) at ($(1)!0.5!(3)$) [scale =2]{$2$};
			\node[draw,circle,minimum size=.5cm,inner sep=1pt] (5) at ($(1)!0.5!(4)$) [scale =2]{$5$};
			\node[draw,circle,minimum size=.5cm,inner sep=1pt] (8) at (0,0) [scale =2]{$8$};
			\node[draw,circle,minimum size=.5cm,inner sep=1pt] (6) at ($(3)!0.5!(8)$) [scale =2]{$6$};
			\node[draw,circle,minimum size=.5cm,inner sep=1pt] (7) at ($(4)!0.5!(8)$) [scale =2]{$7$};

			\draw [line width=2pt,-] (1) -- (2);
			\draw [line width=2pt,-] (1) -- (5);
			\draw [line width=2pt,-] (1) -- (8);
			\draw [line width=2pt,-] (2) -- (3);
			\draw [line width=2pt,-] (3) -- (4);
			\draw [line width=2pt,-] (3) -- (6);
			\draw [line width=2pt,-] (4) -- (5);
			\draw [line width=2pt,-] (4) -- (7);
			\draw [line width=2pt,-] (6) -- (8);
			\draw [line width=2pt,-] (7) -- (8);

			\end{tikzpicture}
		\end{minipage}
		\begin{minipage}[c]{0.4\textwidth}
			\centering
			\small
			\begin{equation*}
			\begin{gathered}
				(1, 2, 3), \ (1, 5, 4),\ (1, 8, 6),\ (1, 8, 7),\ (2, 1, 8), \\
                 (2, 3, 4),\ (2, 3, 4, 5),\ (2, 1, 8, 6),\ (2, 1, 8, 7), \\
                 (3, 4, 5),\  (3, 4, 7),\ (3, 2, 1, 8),\ (4, 3, 6),\   (4, 7, 8),\\
                 (5, 1, 8, 6),\ (5, 4, 7),\ (5, 1, 8),\ (6, 3, 4, 7)
			\end{gathered}
			\end{equation*}
		\end{minipage}
		\\
		\vspace{5mm}
		\begin{minipage}{\textwidth}
			\small
			\begin{equation*}
			\begin{split}
                    w_{1,2} + w_{1,8} + w_{7,8}&< w_{2, 3} + w_{3,4} +w_{4,7}\\
                    w_{1,5} + w_{1,8} + w_{6,8}& < w_{4,5} + w_{3,4} + w_{3,6} \\
                   w_{2,3} + w_{3,4} + w_{4,5} & < w_{1,2} + w_{1,5}  \\
                   w_{3,6} + w_{3,4} + w_{4,7}& < w_{6,8} + w_{7,8} \\
                  w_{4,7} + w_{7,8}& < w_{4, 5} + w_{1,5} +w_{1, 8} \\
                  w_{1,5} + w_{4,5} & < w_{1, 8} +w_{7,8}+w_{4,7}  
			\end{split}
			\end{equation*}
		\end{minipage}
		\begin{center}
			\line(1,0){400}
		\end{center}
	\begin{minipage}[c]{0.4\textwidth}
		\centering
		\begin{tikzpicture}[scale=0.38, every node/.style={scale=0.38}]
		\def \r{3.5}
		\def \s{3.25}
		\def \t{1.5}
		
		\node[draw,circle,minimum size=.5cm,inner sep=1pt] (1) at (-\s,\r) [scale=2]{$1$};
		\node[draw,circle,minimum size=.5cm,inner sep=1pt] (2) at (-\s - \t,0) [scale=2]{$2$};
		\node[draw,circle,minimum size=.5cm,inner sep=1pt] (3) at (-\s,-\r) [scale=2]{$3$};
		\node[draw,circle,minimum size=.5cm,inner sep=1pt] (4) at (-\s+\t,0) [scale=2]{$4$};
		\node[draw,circle,minimum size=.5cm,inner sep=1pt] (5) at (\s,\r) [scale=2]{$5$};
		\node[draw,circle,minimum size=.5cm,inner sep=1pt] (6) at (\s - \t,0) [scale=2]{$6$};
		\node[draw,circle,minimum size=.5cm,inner sep=1pt] (7) at (\s,-\r) [scale=2]{$7$};
		\node[draw,circle,minimum size=.5cm,inner sep=1pt] (8) at (\s+\t,0) [scale=2]{$8$};

		\draw [line width=2pt,-] (1) -- (2);
		\draw [line width=2pt,-] (1) -- (4);
		\draw [line width=2pt,-] (1) -- (5);
		\draw [line width=2pt,-] (2) -- (3);
		\draw [line width=2pt,-] (3) -- (4);
		\draw [line width=2pt,-] (3) -- (7);
		\draw [line width=2pt,-] (5) -- (6);
		\draw [line width=2pt,-] (5) -- (8);
		\draw [line width=2pt,-] (6) -- (7);
		\draw [line width=2pt,-] (7) -- (8);

		\end{tikzpicture}
	\end{minipage}
	\begin{minipage}[c]{0.4\textwidth}
		\centering
		\small
		\begin{equation*}
		\begin{gathered}
			(1, 2, 3), \  (1, 5, 6),\ (1, 2, 3, 7),  \ (1, 5, 8), \ (2, 3, 4),  \\ 
			  (2, 1, 5),  \  (2, 3, 7, 6),\ (2, 3, 7), \ (2, 1, 5, 8), \\
			  (3, 7, 6, 5),\ (3, 7, 6), \   (3, 7, 8), \ (4, 1, 5), \\
			  (4, 1, 5, 6), \  (4, 3, 7), \ (4, 3, 7, 8), \ (5, 6, 7),  \ (6, 5, 8)
		\end{gathered}
		\end{equation*}
	\end{minipage}
	\\
	\vspace{5mm}
	\begin{minipage}{\textwidth}
		\small
		\begin{equation*}
		\begin{split}
			w_{2,3} + w_{3,7} + w_{6,7}& < w_{1,2} + w_{1,5} + w_{5,6}\\
			w_{1,2} + w_{1,5} + w_{5,8}& < w_{2,3} + w_{3,7} + w_{7,8}\\
			w_{1,4} + w_{1,5} + w_{5,6}& < w_{3,4} + w_{3,7} + w_{6,7}\\
			w_{3,4} + w_{3,7} + w_{7,8}& < w_{1,4} + w_{1,5} + w_{5,8}
		\end{split}
		\end{equation*}
	\end{minipage}
	\begin{center}
		\line(1,0){400}
	\end{center}
	\begin{minipage}[c]{0.4\textwidth}
		\centering
		\begin{tikzpicture}[scale=0.35, every node/.style={scale=0.35}]

		\node[draw,circle,minimum size=.5cm,inner sep=1pt] (1) at (0*360/6 +180: 5cm) [scale =2] {$1$};
		\node[draw,circle,minimum size=.5cm,inner sep=1pt] (2) at (3*360/6 +180: 5cm) [scale =2] {$2$};
		\node[draw,circle,minimum size=.5cm,inner sep=1pt] (3) at (1*360/6 +180: 5cm) [scale =2] {$3$};
		\node[draw,circle,minimum size=.5cm,inner sep=1pt] (4) at (2*360/6 +180: 5cm)  [scale =2]{$4$};
		\node[draw,circle,minimum size=.5cm,inner sep=1pt] (5) at (5*360/6 +180: 5cm)  [scale =2]{$5$};
		\node[draw,circle,minimum size=.5cm,inner sep=1pt] (6) at (4*360/6 +180: 5cm)  [scale =2]{$6$};
		\node[draw,circle,minimum size=.5cm,inner sep=1pt] (7) at ($(1)!.33!(2)$)  [scale =2]{$7$};
		\node[draw,circle,minimum size=.5cm,inner sep=1pt] (8) at ($(1)!.66!(2)$)  [scale =2]{$8$};

		\draw [line width=2pt,-] (1) -- (3);
		\draw [line width=2pt,-] (1) -- (5);
		\draw [line width=2pt,-] (1) -- (7);
		\draw [line width=2pt,-] (2) -- (4);
		\draw [line width=2pt,-] (2) -- (6);
		\draw [line width=2pt,-] (2) -- (8);
		\draw [line width=2pt,-] (3) -- (4);
		\draw [line width=2pt,-] (5) -- (6);
		\draw [line width=2pt,-] (7) -- (8);

		\end{tikzpicture}
	\end{minipage}
	\begin{minipage}[c]{0.4\textwidth}
		\centering
		\small
		\begin{equation*}
		\begin{gathered}
		(1,7,8,2), \  (1,3,4),  \ (1,5,6),  \ (1,7,8),\ (2,4,3), \\
		  (2,6,5),\ (2,8,7), \  (3,1,5), \ (3,1,5,6), \ (3,1,7), \\
		   (3,4,2,8), \ (4,2,6,5), \ (4,2,6), \ (4,3,1,7), \ (4,2,8), \\
		    (5,1,7),  \ (5,1,7,8), \ (6,2,8,7), \ (6,2,8) 
		\end{gathered}
		\end{equation*}
	\end{minipage}
	\\
	\vspace{5mm}
	\begin{minipage}{\textwidth}
		\small
		\begin{equation*}
		\begin{split}
		w_{2,6} + w_{2,8} + w_{7,8} & < w_{5,6} + w_{1,5} + w_{1,7}\\
		w_{3,4} + w_{1,3} + w_{1,7} & < w_{2,4} + w_{2,8} + w_{7,8}\\
		w_{1,5} + w_{1,7} + w_{7,8}  & < w_{5,6} + w_{2,6} + w_{2,8}\\
		w_{2,4} + w_{2,6} + w_{5,6}  & < w_{3,4} + w_{1,3} + w_{1,5}\\
		w_{1,3} + w_{1,5} + w_{5,6}  & < w_{3,4} + w_{2,4} + w_{2,6}\\
		w_{3,4} + w_{2,4} + w_{2,8}  & < w_{1,3} + w_{1,7} + w_{7,8} 
		\end{split}
		\end{equation*}
	\end{minipage}
	\begin{center}
		\line(1,0){400}
	\end{center}
\newpage 
\printbibliography
\end{document}